\documentclass[12pt, reqno]{amsart}
\usepackage{amsmath}
\usepackage{amssymb,amsthm,amsmath}
\usepackage[numbers,sort&compress]{natbib}
\usepackage{amssymb,amsmath}
\usepackage{amsfonts}
\usepackage{mathrsfs}
\usepackage{latexsym}
\usepackage{amssymb}
\usepackage{amsthm}
\usepackage{indentfirst}
\date{%\today
January 2, 2023}

\let\oldsection\section
\renewcommand\section{\setcounter{equation}{0}\oldsection}

\newtheorem{corollary}{Corollary}[section]
\newtheorem{theorem}{Theorem}[section]
\newtheorem{lemma}{Lemma}[section]
\newtheorem{proposition}{Proposition}[section]

\newtheorem{remark}{Remark}[section]

\allowdisplaybreaks
\begin{document}

\title[Unboundedness of entropy and uniform positivity of temperature]{Instantaneous unboundedness of the entropy and uniform positivity of the temperature for the compressible Navier-Stokes equations with fast decay density}

\author{Jinkai~Li}
\address[Jinkai~Li]{South China Research Center for Applied Mathematics and Interdisciplinary Studies,
School of Mathematical Sciences, South China Normal University, Guangzhou 510631, China}
\email{jklimath@m.scnu.edu.cn; jklimath@gmail.com}

\author{Zhouping Xin}
\address[Zhouping Xin]{
The Institute of Mathematical Sciences, The Chinese University of Hong Kong, Hong Kong, China}
\email{zpxin@ims.cuhk.edu.hk}


%%% ----------------------------------------------------------------------
\begin{abstract}
This paper concerns the physical behaviors of any solutions to the one dimensional compressible Navier-Stokes equations for viscous and heat conductive gases with constant viscosities and heat conductivity for
fast decaying density at far fields only.
First, it is shown that the specific entropy becomes not uniformly bounded immediately after the initial time,
as long as the initial density
decays to vacuum at the far field at the rate not slower than $O\left(\frac1{|x|^{\ell_\rho}}\right)$ with $\ell_\rho>2$.
Furthermore, for faster decaying initial density, i.e., $\ell_\rho\geq4$, a sharper result is discovered that the absolute temperature becomes uniformly positive at each positive time, no matter whether it is uniformly positive or not initially, and consequently
the corresponding entropy behaves
as $O(-\log(\varrho_0(x)))$ at each positive time, independent of the boundedness of the initial entropy.
Such phenomena
are in sharp contrast to the case with slowly decaying initial density of the rate no faster than $O(\frac1{x^2})$,
for which our previous works
\cite{LIXINADV,LIXINCPAM,LIXIN3DK}
show that the uniform boundedness of the entropy can be propagated for all positive
time and thus the temperature decays to zero at the far
field. These give a complete answer to the problem concerning the propagation of uniform
boundedness of the entropy for the heat conductive ideal gases and, in particular,
show that the algebraic decay rate $2$ of the initial density at the far field is sharp for the uniform
boundedness of the entropy. The tools to prove our main results are based on some scaling transforms, including the
Kelvin transform,
and a Hopf type lemma
for a class of degenerate equations with possible unbounded coefficients.
\end{abstract}

%%% ----------------------------------------------------------------------
\maketitle

%\tableofcontents
\allowdisplaybreaks

\section{Introduction}
The compressible Navier--Stokes equations for the ideal viscous and heat conductive gases read as
\begin{eqnarray}
  \partial_t\rho+\text{div}\,(\rho u)=0, \label{eqr}\\
  \rho(\partial_tu+(u\cdot\nabla)u)-\mu\Delta u-(\mu+\lambda)\nabla\text{div}\,u+\nabla p=0,\label{equ}\\
  c_v\rho(\partial_t\theta+u\cdot\nabla\theta)+p\text{div}\,u-\kappa\Delta\theta=\mathscr Q(\nabla u), \label{eqt}
\end{eqnarray}
where the unknowns $\rho\geq0$, $u\in\mathbb R^N$, with $N$ the spatial dimension,
$\theta\geq0$, and $p=R\rho\theta$, respectively,
represent the density, velocity, temperature, and pressure. Here, $R$ and
$c_v$ are positive constants, $\mu$ and $\lambda$ are
the viscous coefficients, both assumed to be constants and satisfy the physical constraints
$\mu>0$ and $2\mu+N\lambda>0,$
$\kappa$ is the heat conductive coefficient, assumed to be a positive constant, and $\mathscr Q(\nabla u)$ is a quadratic term of $\nabla u$ given as
$$
\mathscr Q(\nabla u)=\frac\mu2|\nabla u+(\nabla u)^T|^2+\lambda(\text{div}\,u)^2.
$$

By the Gibbs equation $\theta Ds=De+pD(\frac1\rho)$, where $s$ is
the specific entropy and $e=c_v\theta$ is the specific internal energy, it holds that
$p=Ae^{\frac s{c_v}}\rho^\gamma$
for some positive constant $A$, where $\gamma-1=\frac R{c_v}$. It is clear that $\gamma>1$. In terms of $\rho$ and $\theta$, the specific entropy $s$ can be expressed as
\begin{equation}\label{ENTROPY}
s=c_v\left(\log\frac RA+\log\theta-(\gamma-1)\log\rho\right),
\end{equation}
satisfying
\begin{equation}
  \label{EQs}
  \rho(\partial_ts+u\cdot\nabla s)-\frac{\kappa}{c_v}\Delta s=\kappa(\gamma-1)\text{div}\left(\frac{\nabla\rho}{\rho}\right)
  +\frac1\theta\left(\mathcal Q(\nabla u)+\kappa\frac{|\nabla\theta|^2}{\theta}\right),
\end{equation}
in the region where both $\rho$ and $\theta$ are positive.

As the governing system in the gas dynamics, the compressible Navier--Stokes equations have been studied
extensively. One of the central concepts in the mathematical theory for the
compressible Navier--Stokes equations is the
vacuum, which, if occurs, means that the density vanishes at either some interior points or on the
boundary or at the far fields. Indeed, the possible presence of vacuum is one of the main difficulties in the theory
of global well-posedness of general solutions to
the compressible Navier--Stokes equations. Note that the equation (\ref{EQs}) for the entropy is highly
degenerate and singular near the vacuum, it is
even more difficult to analyze the dynamic behavior of the entropy in the presence of vacuum.
Due to this, most of the mathematical theories developed in the existing literatures
on the compressible Navier--Stokes equations in the presence
of vacuum are for system (\ref{eqr})--(\ref{eqt}) regardless of the entropy.

There are extensive literatures on the mathematical studies concerning
the compressible Navier--Stokes equations (\ref{eqr})--(\ref{eqt}).
In the one-dimensional case, the corresponding theory is satisfactory and in particular
the global well-posedness has been known for long time.
In the absence of vacuum, for which the information of the entropy follows from that of
the density and the temperature directly by (\ref{ENTROPY}), the
global well-posedness of
strong solutions was established by Kazhikov--Shelukin \cite{KAZHIKOV82} and Kazhikov \cite{KAZHIKOV77}, which
were later extended in the setting of weak solutions, see, e.g., \cite{CHEHOFTRI00,JIAZLO04,ZLOAMO97,ZLOAMO98};
large time behavior of solutions with general initial data was proved by
Li--Liang \cite{LILIANG16}. In the presence of vacuum, but without considering the entropy,
the corresponding global well-posedness were established by the first author of this paper
in \cite{LJK1DHEAT,LJK1DNONHEAT}, for both heat conductive and non-heat conductive ideal gases.
As shown by Hoff--Smoller \cite{HS}, for the one-dimensional compressible Navier--Stokes equations,
no vacuum can be formed later in finite time from non-vacuum initial data,
while such a result remains open in the multidimensional case.

In the multi-dimensional case, the mathematical theory for the compressible
Navier--Stokes equations is less complete than that in the one-dimensional case.
The breakthrough for the global existence of finite energy weak solutions with general initial data and
possible vacuum, to the isentropic compressible
Navier--Stokes equations, was achieved by Lions
\cite{LIONS98,LIONS93}. The results of Lions \cite{LIONS98,LIONS93}
were later improved by Feireisl--Novotn\'y--Petzeltov\'a \cite{FEIREISL01},
Jiang--Zhang \cite{JIAZHA03}, and more recently Bresch--Jabin \cite{BRESCH18}.
For the full compressible Navier--Stokes equations, the global existence of variational weak solutions was proved
by Feireisl \cite{FEIREISL04B}, under some assumptions on the
equations of states.
The uniqueness of weak solutions is still a challenging open problem.
If the initial datum is suitably regular, then the compressible Navier--Stokes equations admit a unique local strong
or classic solution, see \cite{NASH62,SERRIN59,ITAYA71,VOLHUD72,TANI77,VALLI82,LUKAS84} for the
case in the absence of vacuum, and \cite{SALSTR93,CHOKIM04,CHOKIM06-1,CHOKIM06-2,GLLZNOCOM,HUANGNOCOM,LIZHENGNOCOM} for the case in the presence of vacuum. However, the corresponding global existence with general initial
data may not be expected, due to the recent finite time blow up results by Merle--Rapha'el--Rodnianski--Szeftel \cite{MRRSI,MRRSII},
where for the three-dimensional isentropic compressible Navier--Stokes
equations with spherical symmetry,
regular solutions with finite time singularities are constructed for a class of initial data with far field vacuum.
Indeed, up to now, global strong or classical solutions are established only under some additional conditions
on the initial data: the case with small
perturbed initial data around non-vacuum equilibriums was achieved by
Matsumura--Nishida \cite{MATNIS80,MATNIS81,MATNIS82,MATNIS83}, and later developed in many works, see, e.g.,
\cite{PONCE85,VALZAJ86,DECK92,HOFF97,KOBSHI99,DANCHI01,CHENMIAOZHANG10,CHIDAN15,DANXU18,FZZ18}; while the case with
initial data of small energy but
allowing large oscillations and vacuum was proved by Huang--Li--Xin \cite{HLX12} and Li--Xin \cite{LIXIN13} for the
isentropic system, and later generalized to the full system in \cite{HUANGLI11,WENZHU17,LJK3DHEATSMALL}.

It is worth pointing out that
there are some significant differences in the mathematical theories for the compressible Navier--Stokes
equations between the vacuum and non-vacuum cases and new phenomena may occur depending on the locations and states of vacuum.
In the non-vacuum case, the solutions can be establish in both the homogeneous and inhomogeneous spaces depending on the
properties of the initial data, and the
solution spaces guarantee the uniform boundedness of the entropy. However, these may fail in general in the presence of vacuum. Indeed,
in the case that the density has compact support,
the solution can be established in the homogeneous spaces, see, e.g., \cite{CHOKIM04,CHOKIM06-1,CHOKIM06-2,HLX12,GLLZNOCOM,HUANGNOCOM,LIZHENGNOCOM}, but not in the inhomogeneous spaces, see Li--Wang--Xin \cite{LWX}. Further more, the blowup results of Xin \cite{XIN98}
and Xin--Yan \cite{XINYAN13} imply that the global solutions established
in \cite{HUANGLI11,WENZHU17,LJK3DHEATSMALL} must have unbounded entropy, if initially there is an isolated mass group
surrounded by the vacuum region.
However, it is somewhat surprising that if the initial density vanishes only at far fields with a rate no more than
$O(\frac{1}{|x|^2})$, then, as for the non-vacuum case,
the solutions can be established in both the homogeneous and inhomogeneous spaces, and the
entropy can be uniformly bounded, see the recent works
by the authors \cite{LIXINADV,LIXINCPAM,LIXIN3DK}.

It should be noted that since system (\ref{eqr})--(\ref{eqt}) is already
closed, one can indeed establish self-contained mathematical theories for it, as already developed in the previous
works mentioned above. However, since the second law of the thermodynamics
is not taken in to account, these theories are insufficient from the physical point of view. Therefore,
some new theories are needed to provide information for the entropy in the presence of vacuum
to meet the physical requirements. However,
due to the lack of the expression and high singularity and degeneracy of the governing equation for the entropy
near the vacuum region, in spite of its importance,
the mathematical analysis of the entropy for the viscous compressible fluids in the presence of vacuum was rarely carried
out before. Mathematical studies towards this direction has been
initiated in our previous works \cite{LIXINADV,LIXINCPAM} and further developed in \cite{LIXIN3DK},
where the propagation of the uniform boundedness
of the entropy and the inhomogeneous Sobolev regularities was achieved for the compressible Navier--Stokes equations, with or without heat
conductivities, in the presence of vacuum at the far fields,
under the crucial condition that the initial density decays to vacuum at the rate no faster than $O(\frac1{|x|^2})$.

In this paper,
we continue our studies on the dynamic behavior of the entropy in the presence of vacuum.
Different from the cases considered in \cite{LIXINADV,LIXINCPAM,LIXIN3DK}, where the density decays slowly to the vacuum at far fields,
in the current paper, we investigate the case with fast decaying density
at the far fields. For simplicity, we study the one-dimensional case in the current paper
while leave the multi-dimensional case as future works.
It will be shown in this paper that, in sharp contrast to the cases with slowly
decaying density in \cite{LIXINADV,LIXINCPAM,LIXIN3DK}, the uniform boundedness of the entropy can not be propagated
by the compressible Navier--Stokes equations for viscous and heat conductive ideal gases with constant
viscosities and heat conductivities, if the initial density decays faster than the order $O(\frac1{|x|^{\ell_\rho}})$
at the far fields with $\ell_\rho>2$. Since the uniform boundedness of the entropy has already been established in \cite{LIXINADV,LIXINCPAM,LIXIN3DK}
if the decay rate is less than $O(\frac1{|x|^2})$, our results in this paper
reveal that the decay rate $2$ of the initial density at the far field is sharp for the uniform boundedness of the
entropy. Surprisingly,
in case that the initial density decays faster than the order $O(\frac1{x^4})$, some sharper results can be achieved:
the temperature is uniformly positive immediately after the initial time, for any general nonnegative
(not identically zero) initial temperature, and, as a result, the entropy tends to infinity at the order $O(-\log(\varrho_0(x)))$
at any positive time.

%Recently, we have initiated studies on these issues in the one dimensional case in \cite{LIXINADV,LIXINCPAM}.
%It was proved in \cite{LIXINADV,LIXINCPAM} that the one-dimensional compressible Navier--Stokes
%equations, with or without heat conducting,
%can propagate the uniform boundedness of the entropy locally or globally in time, as long as the initial density
%vanishes only at far fields with a rate no more than $O(\frac{1}{x^2})$. However, the problems in the multi-dimensional
%case have not been studied.

%In this paper, we continue our studies, initiated in \cite{LIXINADV,LIXINCPAM},
%on the uniform boundedness of the entropy and well-posedness of strong solutions in inhomogeneous spaces
%for the multi-dimensional full compressible Navier--Stokes equations
%in the presence of vacuum. We will focus on the heat conductive flows. Note that
%for the heat conductive case one only need to deal with the
%the far field vacuum, as the heat conductivity makes the temperature strictly
%positive everywhere after the initial time, which implies that the entropy becomes unbounded instantaneously
%if the interior vacuum occurs initially. It is noted that the problem of
%the existence of solutions in the inhomogeneous Sobolev spaces, under some
%conditions on the initial density allowing vacuum at the far fields has been studied in \cite{LIXINADV} for the non-heat conductive
%compressible flows in one dimension, but it has not yet been studied either for the heat conductive flows or in multi dimensions.

Consider the Cauchy problem to the one-dimensional compressible Navier--Stokes equations for viscous and heat conductive
ideal gases
\begin{eqnarray}
  \rho_t+(\rho u)_x&=&0, \label{eqr1d}\\
  \rho(u_t+uu_x)-\mu u_{xx}+p_x&=&0,\label{equ1d}\\
  c_v\rho(\theta_t+u\theta_x)+pu_x-\kappa\theta_{xx}&=&\mu(u_x)^2, \label{eqt1d}
\end{eqnarray}
where $p=R\rho\theta$, subject to the initial condition
\begin{equation}
  (\rho, u, \theta)|_{t=0}=(\rho_0, u_0, \theta_0). \label{IC-1D}
\end{equation}

The main results of this paper will be stated and proved in the Lagrangian coordinates; however,
since the velocity of the solutions obtained in this paper have Lipschitz regularities in the
spatial variable, the results
can be transformed back to those in the Eulerian coordinates.

Define the coordinate transform
between the Lagrangian coordinate $y$ and the Eulerian coordinate $x$ as
$x=\eta(y,t)$ satisfying
\begin{equation*}\label{flowmap}
  \left\{
  \begin{array}{l}
  \partial_t\eta(y,t)=u(\eta(y,t),t),\\
  \eta(y,0)=y.
  \end{array}
  \right.
\end{equation*}
Set
\begin{equation*}
  \varrho(y,t):=\rho(\eta(y,t),t),\quad v(y,t):=u(\eta(y,t),t), \quad \vartheta(y,t):=\theta(\eta(y,t),t), \label{newunknown}
\end{equation*}
and
\begin{equation*}
 J:= J(y,t)=\eta_y(y,t).
\end{equation*}
Then, it holds that
\begin{equation*}
J_t=v_y,\quad J|_{t=0}\equiv1, \quad J\varrho=\varrho_0,\label{LCNSJ}
\end{equation*}
with $\varrho_0:=\rho_0$. We still use $s$ to denote the specific entropy in the Lagrangian coordinates. Then, it follows from (\ref{ENTROPY}) that
\begin{equation}\label{ENTROPY'}
s(y,t)=c_v\left(\log\frac RA+\log\vartheta(y,t)-(\gamma-1)\log\varrho_0(y)+(\gamma-1)\log J(y,t)\right),
\end{equation}
for any $y\in\mathbb R$ and $t\in[0,\infty)$.

Then, in the Lagrangian coordinates, the system (\ref{eqr1d})--(\ref{eqt1d}) becomes
\begin{eqnarray}
  J_t&=&v_y,\label{EqJ}\\
  \varrho_0v_t-\mu\left(\frac{v_y}{J}\right)_y+\pi_y&=&0,\label{Eqv}\\
  c_v\varrho_0\vartheta_t+v_y\pi-\kappa\left(\frac{\vartheta_y}{J}\right)_y&=&\mu\frac{|v_y|^2}{J},\label{Eqtheta}
\end{eqnarray}
where $\pi=R\frac{\varrho_0}{J}\vartheta$.
The initial data can be taken as
\begin{equation}
  \label{IC0}
  (J,v,\vartheta)|_{t=0}=(1,v_0,\vartheta_0),
\end{equation}
where $v_0=u_0$ and $\vartheta_0=\theta_0$.

The following conventions will be used throughout this paper.
For $1\leq q\leq\infty$ and positive integer $m$, $L^q=L^q(\mathbb R)$ and
$W^{1,q}=W^{m,q}(\mathbb R)$ denote the standard Lebesgue and Sobolev spaces,
respectively, and $H^m=W^{m,2}$. For simplicity, $L^q$ and
$H^m$ denote also their $N$ product spaces $(L^q)^N$ and $(H^m)^N$, respectively.
$\|u\|_q$ is the $L^q$ norm of $u$, and $\|(f_1,f_2,\cdots,f_n)\|_X$ is the sum
$\sum_{i=1}^N\|f_i\|_X$ or the equivalent norm $\left(\sum_{i=1}^N\|f_i\|_X^2
\right)^{\frac12}$.

The main results of this paper are the following three theorems. The first one yields the global existence of a solution to the Cauchy problem (\ref{EqJ})--(\ref{Eqtheta}), subject to (\ref{IC0}).

\begin{theorem}
  \label{THMGLOBAL}
  Let the initial density $\varrho_0$ be given such that $0<\varrho_0\in L^1(\mathbb R)\cap W^{2,\infty}(\mathbb R)$ and
  \begin{equation}
  |\varrho_0'|+|\varrho_0''|\leq K_1\varrho_0\quad \mbox{ on }\mathbb R,\tag{H1}
\end{equation}
for a positive constant $K_1$.
  Assume that $(v_0, \vartheta_0)$ satisfies $\vartheta_0\geq0$ on $\mathbb R$ and
  \begin{eqnarray}
    &\displaystyle(\sqrt{\varrho_0}v_0, \sqrt{\varrho_0}v_0^2, v_0', v_0'', \sqrt{\varrho_0}\vartheta_0, \sqrt{\varrho_0}\vartheta_0', \sqrt{\varrho_0}\vartheta_0'')\in L^2(\mathbb R), \quad\frac{G_0'}{\sqrt{\varrho_0}}\in L^2(\mathbb R), \\
    &\displaystyle\varliminf_{y\rightarrow-\infty}\frac{|v_0'(y)|}{\sqrt{\varrho_0(y)}}+\varliminf_{y\rightarrow+\infty}\frac{|v_0'(y)|}
    {\sqrt{\varrho_0(y)}}<+\infty,\label{H3}
  \end{eqnarray}
  where $G_0:=\mu v_0'-R\varrho_0\vartheta_0$.

  Then, there is a global solution $(J, v, \vartheta)$ to (\ref{EqJ})--(\ref{Eqtheta}), subject
  to (\ref{IC0}),
  satisfying $\inf_{(y,t)\in\mathbb R\times(0,T)}J>0$, $\theta\geq0$, and
  \begin{align*}
    &\frac{J_y}{\sqrt{\varrho_0}}, J_{yy}, J_t, J_{yt}\in L^\infty(0,T; L^2(\mathbb R)), \\
    &\sqrt{\varrho_0}v, \sqrt{\varrho_0}v^2, v_y, \frac{v_{yy}}{\sqrt{\varrho_0}}, \sqrt{\varrho_0}v_t\in L^\infty(0,T; L^2(\mathbb R)),\quad
      v_{yyy}, v_{yt}\in L^2(0,T; L^2(\mathbb R)),\\
    &\sqrt{\varrho_0}\vartheta, \sqrt{\varrho_0}\vartheta_y, \sqrt{\varrho_0}\vartheta_{yy}, \varrho_0^\frac32\vartheta_t\in L^\infty(0,T; L^2(\mathbb R)),\quad\vartheta_y\in L^2(0,T; H^2(\mathbb R)),  \\
    &\varrho_0\vartheta_t, \varrho_0\vartheta_{yt}\in L^2(0,T; L^2(\mathbb R)), \quad G_t, \left(\frac{G_y}{\varrho_0}\right)_y\in L^2(0,T; L^2(\mathbb R)),
  \end{align*}
  for any positive time $T$, where $G:=\mu\frac{v_y}{J}-R\frac{\varrho_0}{J}\vartheta$.
\end{theorem}

\begin{remark}
  (i) Condition (H1) allows arbitrary algebraic and even exponential decay rate of $\varrho_0$ at far fields.
  Indeed, one can check that functions of the forms $\frac{A}{(1+y^2)^\ell}$ and $e^{-(1+y^2)^\delta}$, with
  $A,\ell\in(0,\infty)$ and $\delta\in(0,\frac12]$, satisfy (H1).
  Thus, Theorem \ref{THMGLOBAL} generalizes the global existence result in our previous work \cite{LIXINCPAM},
  where some assumptions on slow decay at far fields on $\varrho_0$ are assumed.

  (ii) Condition (\ref{H3}) is used only to construct suitable approximated initial data for the corresponding initial boundary value
  problems (which are expected to converge to the Cauchy problem), see Step 1 in the proof of Theorem \ref{THMGLOBAL}.
\end{remark}

The second theorem gives the immediate unboundedness of the specific entropy if the algebraic decay rate of the initial density is greater than $2$.

\begin{theorem}
\label{UNBDDENTROPY}
  Assume, in addition to the conditions in Theorem \ref{THMGLOBAL}, that
\begin{equation}
  (1+|y|)^{\ell_\rho}\varrho_0(y)\leq K_2 , \quad\forall y\in\mathbb R, \tag{H2}
\end{equation}
for some positive constants $\ell_\rho\in(2,\infty)$ and $K_2$, and either
$\vartheta_0$ is not identically zero or $v_0$ is not identically a constant.
Let $(J, v, \vartheta)$ be a solution
to system (\ref{EqJ})--(\ref{Eqtheta}), subject to (\ref{IC0}), satisfying the properties
stated in Theorem \ref{THMGLOBAL}. Then, the specific entropy $s\not\in L^\infty(\mathbb R\times(0,T))$, for any positive time $T\in(0,\infty)$.
\end{theorem}

\begin{remark}
 Theorem \ref{UNBDDENTROPY} reveals a completely different phenomenon from that
  in \cite{LIXINADV,LIXINCPAM,LIXIN3DK}, where the initial density decays no faster than $O(\frac1{y^2})$ at far fields, so that the entropy keeps uniformly bounded. While Theorem \ref{UNBDDENTROPY}
  shows that if the initial density decays faster than $O\left(\frac1{|y|^{\ell_\rho}}\right)$, with $\ell_\rho>2$, at far fields, then
  the entropy becomes not uniformly bounded immediately after the initial time.
 % Therefore, the algebraic decaying rate $2$ at the far field of the initial density
%  is sharp for the propagation of uniform boundedness of entropy for the ideal gases in the presence of far field vacuum.
Consequently, we have given a complete answer to the problem concerning the propagation of
uniform boundedness of entropy for ideal gases in one dimension: the uniform boundedness of the entropy for
the ideal gases, in the presence of vacuum at the far fields only in one dimension,
can be propagated if and only if the algebraic decay rate of the initial density is not greater than $2$. In other words, the decay
rate $2$ of the initial density at the far fields is sharp for the uniform boundedness of the entropy in one dimension.
\end{remark}

The main ingredients of the proof of Theorem \ref{UNBDDENTROPY} are based on using some scaling transform to transform the far field vacuum to an interior vacuum and applying a Hopf type lemma for a class of linear degenerate elliptic
equations with degeneracy in the time variable and possible unbounded coefficients.
The scaling transform for the temperature to be used here is
$$
f(y,t):=\vartheta(y^{-\beta},t),\quad y\in(0,\infty), t\in[0,\infty),
$$
for some suitably chosen $\beta>0$. Similar transform can also be introduced for negative $y$.
Due to the continuity equation (\ref{eqr1d}) and the
assumption that the initial density reaches vacuum only at the far fields, the density remains positive on any compact interval for all
positive time. Thus the equation (\ref{eqt1d}) can be regarded a uniform
parabolic equation for $\theta$ on compact domains. Consequently, the temperature will be positive on any finite interval for any positive time
$t$ by the strong maximum principle, and thus $f$ is positive for any positive $y$ and $t$. By using the properties of $\vartheta$ stated
in Theorem \ref{THMGLOBAL}, one can verify that $0<f\in C^{2,1}((0,\infty)\times(0,\infty))$. Assuming by contradiction that
the entropy is uniformly
bounded, one can extend $f$ by zero on the positive time axis,
such that $0\leq f\in C([0,\infty)\times[0,\infty))$ and reaches zero on the positive time axis only. The temperature equation yields
\begin{equation*}
  a_0f_t-af_{yy}+bf_y+\tilde cf\geq0, \quad\mbox{in }(0,\infty)\times(0,\infty),
\end{equation*}
which motivates us to apply the Hopf type lema to $f$ at the points on the positive time axis.
By choosing $\beta$ suitably, one can verify that
the coefficients $a_0$ and $\tilde c$ are uniformly bounded near the positive time axis;
however, the coefficient $b$ contains an unbounded term involving $\frac1y$. Fortunately, such an unbounded term in $b$ is of ``right" sign
while the remaining term in $b$ is uniformly bounded for suitably chosen $\beta$, so
that the Hopf type lemma still holds (see Lemma \ref{LEMHOPF}). Thus applying
the Hopf type lemma to $f$ near the positive time axis leads to a quantitative asymptotic
behavior of the temperature at the far field. The contradiction comes from the fact that the asymptotic behavior of the temperature
derived from the
Hopf type lemma is not consistent with
that derived from (H2) and the uniform boundedness of the entropy. This inconsistency implies that the entropy can not be uniformly bounded
and thus Theorem \ref{UNBDDENTROPY} follows.

The third theorem gives the uniform positivity of the temperature and consequently the
asymptotic unboundedness of the entropy, which are sharper results than those in Theorem \ref{UNBDDENTROPY}, under the stronger assumption that the algebraic decay rate of the initial density at the far field is greater than $4$.

\begin{theorem}
  \label{THMPOSTEM}
Assume, in addition to the conditions in Theorem \ref{THMGLOBAL}, that
\begin{equation}
  (1+|y|)^4\varrho_0(y)\leq K_3, \quad\forall y\in\mathbb R, \tag{H3}
\end{equation}
for a positive constant $K_3$, and either
$\vartheta_0$ is not identically zero or $v_0$ is not identically a constant.
Let $(J, v, \vartheta)$ be a solution
to system (\ref{EqJ})--(\ref{Eqtheta}), subject to (\ref{IC0}), satisfying the properties stated in Theorem \ref{THMGLOBAL}.

Then, the following statements hold:

(i) the temperature $\vartheta$ satisfies
$$
\inf_{y\in\mathbb R}\vartheta(y, t)>0,\quad\forall t\in(0,\infty);
$$

(ii) the specific entropy $s$
satisfies
$$
R\leq\varliminf_{|y|\rightarrow\infty}\frac{s(y,t)}{-\log(\varrho_0(y))}\leq\varlimsup_{|y|\rightarrow\infty}
\frac{s(y,t)}{-\log(\varrho_0(y))}<\infty , \quad\forall t\in(0,\infty).
$$
In particular, $s$ becomes unbounded immediately after the initial time,
regardless of whether it is uniformly bounded or not at the initial time.
\end{theorem}

\begin{remark}
  It is an interesting question to show whether Theorem \ref{THMPOSTEM} still holds in the case that the algebraic
  decay rate of $\varrho_0$ lies between $2$ and $4$. However,
  as already shown in Theorem \ref{UNBDDENTROPY}, in this case, though the uniform positivity of the temperature is not clear, yet
  the specific entropy becomes not uniformly bounded in any positive time.
\end{remark}

Recall that
the temperature is positive on any finite interval for any positive time $t$. To obtain the positive lower bound for the temperature at
any positive time, it suffices to achieve this at far fields.
To this end, similar as in the proof of Theorem \ref{UNBDDENTROPY}, we apply some scaling technique to
transform the far field vacuum to an interior vacuum and take advantage of the Hopf type lemma. However, the scaling transform
introduced before does not work here directly. Instead, we apply the
Kelvin transform to the temperature $\vartheta$ and denote by $h$ the transformed temperature, that is,
$$
h(y,t)=y\vartheta\left(\frac1y,t\right),\quad \forall y\not=0, t\in[0,\infty),
$$
which satisfies a linear degenerate equation,
with all coefficients being uniformly bounded by the assumption (H3). By using the properties of $\vartheta$ stated in Theorem
\ref{THMGLOBAL}, one can verify that $0\leq h\in C^{2,1}(\Omega)\cap C(\overline\Omega)$ and more importantly $h(0,t)=0$,
where $\Omega=((-\infty,0)\cup(0,\infty))\times(0,\infty)$. Note that different from the proof of Theorem
\ref{UNBDDENTROPY}, here the important property that
$h(0,t)=0$ holds without any condition on the entropy.
By the
Hopf type lemma (Lemma \ref{LEMHOPF}) and applying the
strong maximum principle,
we can derive that $h$ behaves linearly near the origin at each positive time and hence obtain
the uniformly positive lower bound for the temperature near the far fields. With the
aid of the positive lower bound of the temperature, the
asymptotic unboundedness of the entropy follows from (\ref{ENTROPY'}) as $J$ has uniform positive lower and upper bounds.

The rest of this paper is arranged as follows: in Section \ref{SECIBVP}, we consider a carefully designed initial-boundary value problem
for the system (\ref{EqJ})--(\ref{Eqtheta}) and establish a series of a priori estimates on the solution
independent of the length of the spatial interval;
in Section \ref{SECGLOBAL}, we obtain the global existence of solutions to the Cauchy problem and thus prove
Theorem \ref{THMGLOBAL} by taking limit of the solutions obtained in Section \ref{SECIBVP};
Section \ref{SECUNBDDENTROPY} is devoted to the proof of Theorem \ref{UNBDDENTROPY}; and finally, the proof of Theorem \ref{THMPOSTEM} is given
in Section \ref{SECPOSTEM}.

Throughout this paper, $C$ will denote a generic positive constant which may vary from place to place.

\section{Initial-boundary value problem and a priori estimates}
\label{SECIBVP}
Throughout this section, we consider the initial-boundary value problem to the system (\ref{EqJ})--(\ref{Eqtheta}), in
$(\alpha, \beta)\times(0,\infty)$, with $-\infty<\alpha<\beta<+\infty$, subject to the initial-boundary conditions:
\begin{align}
  (J, v, \vartheta)|_{t=0}=(1, v_0, \vartheta_0), \label{IC}\\
  (v_y, \vartheta)|_{y= \alpha, \beta }=(0, 0). \label{BC}
\end{align}

The following global well-posedness can be proved in the same way as in \cite{LJK1DNONHEAT}.

\begin{proposition}
\label{PROPGLOBAL}
Let $(\varrho_0, v_0, \vartheta_0)\in H^2((\alpha, \beta))$ be given such that $\varrho_0, \vartheta_0\geq0$ on $(\alpha, \beta)$ and
$v_0'(\alpha)=v_0'(\beta)=\vartheta_0(\alpha)=\vartheta_0(\beta)=0.$
Assume that
\begin{eqnarray*}
  \mu v_0''-R(\varrho_0\vartheta_0)'=\sqrt{\varrho_0}g_1, \quad \kappa\vartheta_0''+\mu(v_0')^2-Rv_0'\varrho_0\vartheta_0=\sqrt{\varrho_0}g_2,
\end{eqnarray*}
for two functions $g_1, g_2\in L^2((\alpha, \beta))$.

Then, there is a unique global solution $(J, v, \vartheta)$ to system (\ref{EqJ})--(\ref{Eqtheta}),
in $(\alpha, \beta)\times[0,\infty)$, subject to (\ref{IC})--(\ref{BC}), satisfying
$\inf_{(y,t)\in(\alpha,\beta)\times(0,T)}J>0$, $\vartheta\geq0$, and
\begin{eqnarray*}
  &&J\in C([0,T]; H^2((\alpha, \beta))),\quad J_t\in L^2(0,T;H^2((\alpha, \beta))),\\
  &&v,\vartheta\in C([0,T]; H^2((\alpha, \beta)))\cap L^2(0,T; H^3((\alpha, \beta))),\quad v_t,\vartheta_t\in L^2(0,T; H^1((\alpha, \beta))),
\end{eqnarray*}
for any $T\in(0,\infty)$.
\end{proposition}

The rest of this section is devoted to deriving the a priori estimates, independent of $\alpha$
and $\beta$, on the unique global solution
$(J, v, \vartheta)$ stated in Proposition \ref{PROPGLOBAL}. Keeping this in mind, in the rest of this section, we will always
assume that $(J, v, \vartheta)$ is the solution stated in Proposition \ref{PROPGLOBAL}.

Throughout this section, for simplicity of notations,
the norms $\|\cdot\|_q$ and $\|\cdot\|_{H^1}$ are the corresponding ones on the interval
$(\alpha, \beta)$, that is,
$$
\|\cdot\|_q:=\|\cdot\|_{L^q((\alpha, \beta))}\quad\mbox{and}\quad \|\cdot\|_{H^1}:=\|\cdot\|_{H^1((\alpha,
\beta))}.
$$
Denote
\begin{equation*}
  \label{a0E0}
  m_0:=\int_{\alpha}^{\beta}\varrho_0dy,\quad \mathscr E_0:=\int_{\alpha}^{\beta}\varrho_0\left(\frac{v_0^2}{2}+c_v\vartheta_0\right)dy.
\end{equation*}

\begin{proposition}
\label{BASIC}
It holds that
\begin{equation*}
  \int_{\alpha}^{\beta}\varrho_0\left(\frac{v^2}{2}+c_v\vartheta\right)dy\leq\mathscr E_0.
\end{equation*}
\end{proposition}

\begin{proof}
Multiplying (\ref{Eqv}) with $v$, integrating over $(\alpha, \beta)$, and by the boundary conditions,
one gets by integration by parts that
\begin{equation}
  \frac12\frac{d}{dt}\int_{\alpha}^{\beta}\varrho_0v^2dy+\mu\int_{\alpha}^{\beta}\frac{|v_y|^2}{J}dy-\int_{\alpha}^{\beta}v_y\pi dy=0.\label{BS-1}
\end{equation}
Since $\vartheta\geq0$ in $(\alpha, \beta)\times(0,\infty)$, it is clear that $\vartheta_y(\alpha,t)\geq0$ and $\vartheta_y(\beta,t)\leq0$,
for any $t\in(0,\infty)$. As are result, integrating (\ref{Eqtheta}) over $(\alpha, \beta)$ and integration by parts yield
\begin{equation}
  c_v\frac{d}{dt}\int_{\alpha}^{\beta}\varrho_0\vartheta dy+\int_{\alpha}^{\beta}v_y\pi dy\leq\mu\int_{\alpha}^{\beta}\frac{|v_y|^2}{J}dy.\label{BS-2}
\end{equation}
Summing (\ref{BS-1}) with (\ref{BS-2}) and integrating with respect to $t$ lead to the conclusion.
\end{proof}

\begin{proposition}\label{Prop2.1}
It holds that
$$
 e^{-\frac2\mu\sqrt{2m_0\mathscr E_0}}\leq J\leq e^{\frac4\mu\sqrt{2m_0\mathscr E_0}}\left(1+\frac R\mu\int_0^t\varrho_0\vartheta d\tau\right),\quad\forall t\in(0,\infty).
$$
\end{proposition}

\begin{proof}
Since $v_y|_{y=\alpha}=0$ and $J|_{t=0}=1$, it follows from (\ref{EqJ}) that $J|_{y=\alpha}=1$. Substituting (\ref{EqJ})
into (\ref{Eqv}) yields
$$
\varrho_0v_t-\mu(\log J)_{yt}+\pi_y=0,
$$
from which, integrating over $(0,t)$ and using $J|_{t=0}=1$, one can get
$$
\varrho_0(v-v_0)+\int_0^t\pi_y ds=\mu(\log J)_y.
$$
Integrating this over $(\alpha,y)$ and noticing that $J|_{y=\alpha}=1$ and
$\pi|_{y=\alpha}=R\frac{\varrho_0}{J}\vartheta|_{y=\alpha}=0$, one gets
$$
\int_\alpha^y\varrho_0(v-v_0)dz+\int_0^t\pi ds=\mu\log J,
$$
which leads to
\begin{equation}
  J=e^{\frac1\mu\left(\int_\alpha^y\varrho_0(v-v_0)dz+\int_0^t\pi ds\right) }. \label{BJ-1}
\end{equation}
It follows from Proposition \ref{BASIC} and the H\"older inequality that
\begin{eqnarray}
\int_\alpha^\beta\varrho_0(|v|+|v_0|)dz &\leq&\left(\int_\alpha^\beta\varrho_0 dz\right)^\frac12
  \left[\left(\int_\alpha^\beta\varrho_0v^2dz\right)^\frac12+\left(\int_\alpha^\beta\varrho_0v_0^2dz\right)^\frac12\right]
  \nonumber\\
  &\leq& 2\sqrt{2 m_0\mathscr E_0}. \label{BJ-2}
\end{eqnarray}
With the aid of (\ref{BJ-2}) and since $\pi\geq0$, it follows from (\ref{BJ-1}) that
\begin{equation}
  J\geq e^{-\frac1\mu\int_\alpha^\beta\varrho_0(|v|+|v_0|)dz}\geq e^{-\frac2\mu\sqrt{2m_0\mathscr E_0}}.\label{BJ-3}
\end{equation}
Rewrite (\ref{BJ-1}) as
$Je^{-\frac1\mu \int_\alpha^y\varrho_0(v-v_0)dz }=e^{\frac1\mu \int_0^t\pi ds}$. Thus
$$
\frac1\mu J\pi \exp\left\{-\frac1\mu \int_\alpha^y\varrho_0(v-v_0)dz \right\} =
\partial_t (e^{\frac1\mu \int_0^t\pi ds} ).
$$
Hence, one gets by noticing $J\pi=R\varrho_0\vartheta$ that
\begin{equation*}
\exp\left\{\frac1\mu \int_0^t\pi ds\right\}
 = 1+\frac R\mu\int_0^t\varrho_0\vartheta \exp\left\{-\frac1\mu \int_\alpha^y\varrho_0(v-v_0)dz \right\}ds.
\end{equation*}
Substituting this into (\ref{BJ-1}) and using (\ref{BJ-2}) lead to
\begin{eqnarray*}
  J&=&e^{\frac1\mu \int_\alpha^y\varrho_0(v-v_0)dz}\left(1+\frac R\mu\int_0^t\varrho_0\vartheta \exp\left\{-\frac1\mu \int_\alpha^y\varrho_0(v-v_0)dz \right\}ds\right) \\
  &\leq&e^{\frac4\mu\sqrt{2m_0\mathscr E_0}}\left(1+\frac R\mu\int_0^t\varrho_0\vartheta ds\right).
\end{eqnarray*}
Combining this with (\ref{BJ-3}) yields the conclusion.
\end{proof}

In the rest of this section, we will always assumed that $C$ is a general
positive constant depending only on $R, c_v, \mu, \kappa,
K_1, T$, and the upper bound of $\mathscr N_0$, but independent of $\alpha$ and $\beta$ with $\beta-\alpha\geq1$,
where
\begin{equation}\label{N0}
\mathscr N_0:=\|\varrho_0\|_\infty+m_0+\mathscr E_0+\left\|\left(\sqrt{\varrho_0} v_0^2, v_0', v_0'', \sqrt{\varrho_0}\vartheta_0, \sqrt{\varrho_0}\vartheta_0', \sqrt{\varrho_0}\vartheta_0'',
G_0, \frac{G_0'}{\sqrt{\varrho_0}}\right)\right\|_2.
\end{equation}

\begin{proposition}
  \label{Prop2.2}
  It holds that
  $$
  \sup_{0\leq t\leq T}\|(\sqrt{\varrho_0}v^2, \sqrt{\varrho_0}\vartheta)\|_2^2+\int_0^T\left(
  \|\sqrt{\varrho_0}\vartheta\|_\infty^2+\left\|\frac{vv_y}{\sqrt J}\right\|_2^2+  \left\|\frac{\vartheta_y}{\sqrt J}\right\|_2^2\right)dt\leq C.
  $$
\end{proposition}

\begin{proof}
Set $E=\frac{v^2}{2}+c_v\vartheta$. Then, it follows from (\ref{Eqv}) and (\ref{Eqtheta}) that
$$
\varrho_0E_t-\kappa\left(\frac{\vartheta_y}{J}\right)_y=\left(\mu\frac{vv_y}{J}-R\frac{\varrho_0}{J}\vartheta v\right)_y.
$$
Note that $\vartheta_y(\alpha,t)\geq0$ and $\vartheta_y(\beta,t)\leq0$ due to the boundary condition $\vartheta|_{y=\alpha,\beta}=0$ and
the fact that $\vartheta\geq0$ in $(\alpha, \beta)\times(0,\infty)$. Multiplying the above equation with $E$ and integration by parts yield
\begin{eqnarray*}
  &&\frac12\frac{d}{dt}\|\sqrt{\varrho_0}E\|_2^2+\kappa c_v\int_\alpha^\beta\frac{|\vartheta_y|^2}{J}dy-\kappa E\frac{\vartheta_y}{J}\Big|_{y=\alpha}^\beta\\
&\leq& -\int_\alpha^\beta\left(\mu\frac{vv_y}{J}-R\frac{\varrho_0}{J}\vartheta v\right)(vv_y+c_v\vartheta_y) dy
-\kappa\int_\alpha^\beta\frac{\vartheta_y}{J}vv_ydy\\
  &\leq&\frac{\kappa c_v}2\left\|\frac{\vartheta_y}{\sqrt J}\right\|_2^2+C\int_\alpha^\beta\frac1J\left(|vv_y|^2+\varrho_0^2 v^2\vartheta^2\right)dy,
\end{eqnarray*}
and thus, by the Cauchy inequality and that $-\kappa E\frac{\vartheta_y}{J}\Big|_{y=\alpha}^\beta\geq0$, it follows that
\begin{equation}
  \label{EstE-1}
  \frac{d}{dt}\|\sqrt{\varrho_0}E\|_2^2+\kappa c_v\left\|\frac{\vartheta_y}{\sqrt J}\right\|_2^2
  \leq A_1\left\|\frac{vv_y}{\sqrt J}\right\|_2^2+A_1\int_\alpha^\beta\frac1J\varrho_0^2v^2\vartheta^2dy,
\end{equation}
for a positive constant $A_1$ depending only on $\kappa, c_v, \mu,$ and $R$.
Multiplying (\ref{Eqv}) with $4v^3$, using the boundary conditions, and
integration by parts, one deduces
\begin{eqnarray*}
  \frac{d}{dt}\int_\alpha^\beta\varrho_0 v^4dy+12\mu\left\|\frac{vv_y}{\sqrt J}\right\|_2^2=12R\int_\alpha^\beta
  \frac1Jv^2 v_y\varrho_0\vartheta dy\\
  \leq6\mu\left\|\frac{vv_y}{\sqrt J}\right\|_2^2+C\int_\alpha^\beta\frac1J\varrho_0^2v^2\vartheta^2dy,
\end{eqnarray*}
and thus,
\begin{equation}
  \frac{d}{dt}\int_\alpha^\beta\varrho_0 v^4dy+6\mu\left\|\frac{vv_y}{\sqrt J}\right\|_2^2
  \leq C\int_\alpha^\beta\frac1J\varrho_0^2v^2\vartheta^2dy. \label{EstE-2}
\end{equation}
Multiplying (\ref{EstE-2}) with $\frac{A_1}{3\mu}$ and summing the resultant with (\ref{EstE-1}) yield
\begin{equation*}
  \frac{d}{dt}\left(\|\sqrt{\varrho_0}E\|_2^2+\frac{A_1}{3\mu}\|\sqrt{\varrho_0}v^2\|_2^2\right)
  +\kappa c_v\left\|\frac{\vartheta_y}{\sqrt J}\right\|_2^2+A_1\left\|\frac{vv_y}{\sqrt J}\right\|_2^2
  \leq C\int_\alpha^\beta\frac1J\varrho_0^2v^2\vartheta^2dy,
\end{equation*}
from which, by Proposition \ref{BASIC} and Proposition \ref{Prop2.1}, one gets
\begin{align}
   \frac{d}{dt}\left(\|\sqrt{\varrho_0}E\|_2^2+\frac{A_1}{3\mu}\|\sqrt{\varrho_0}v^2\|_2^2\right)
  +\kappa c_v\left\|\frac{\vartheta_y}{\sqrt J}\right\|_2^2+A_1\left\|\frac{vv_y}{\sqrt J}\right\|_2^2\nonumber\\
  \leq C\|\sqrt{\varrho_0}v\|_2^2\|\sqrt\varrho_0\vartheta\|_\infty^2 \leq C\|\sqrt{\varrho_0}\vartheta\|_\infty^2.\label{2.4}
\end{align}

Since $\vartheta|_{y=\alpha}=0$, it follows from Proposition \ref{Prop2.1}, the H\"older and Young inequalities, and (H1) that
\begin{eqnarray}
  \varrho_0\vartheta^2 &=&\int_\alpha^y(\varrho_0\vartheta^2)_ydz=\int_\alpha^y(\varrho_0'\vartheta^2 +2\varrho_0\vartheta
  \vartheta_y)dz\nonumber\\
  &\leq&\int_{\alpha}^{\beta}\left(K_1\varrho_0\vartheta^2+2\varrho_0\vartheta\frac{\vartheta_y}{\sqrt J}\sqrt J\right)dz\nonumber\\
  &\leq&K_1\|\sqrt{\varrho_0}\vartheta\|_2^2+2\|\varrho_0\vartheta\|_1^{\frac12}
  \|\varrho_0\|_\infty^{\frac14}\|\sqrt{\varrho_0}\vartheta
  \|_\infty^{\frac12}\left\|\frac{\vartheta_y}{\sqrt J}\right\|_2\|J\|_\infty^{\frac12}\nonumber\\
  &\leq&K_1\|\sqrt{\varrho_0}\vartheta\|_2^2+C \|\sqrt{\varrho_0}\vartheta
  \|_\infty^{\frac12}\left\|\frac{\vartheta_y}{\sqrt J}\right\|_2  \left(1+ \int_0^t\|\varrho_0\vartheta\|_\infty d\tau\right)^{\frac12}\nonumber\\
  &\leq&K_1\|\sqrt{\varrho_0}\vartheta\|_2^2+C\|\sqrt{\varrho_0}\vartheta
  \|_\infty^{\frac12}\left\|\frac{\vartheta_y}{\sqrt J}\right\|_2\left[1+\left(\int_0^t\|\varrho_0\vartheta\|_\infty^2 d\tau\right)^{\frac14}\right]\nonumber\\
  &\leq&\frac12\left(\|\sqrt{\varrho_0}\vartheta\|_\infty^2+\epsilon\left\|\frac{\vartheta_y}{\sqrt J}\right\|_2^2\right)
  +C\left(1+\|\sqrt{\varrho_0}\vartheta\|_2^2+\int_0^t\|\sqrt{\varrho_0}\vartheta\|_\infty^2d\tau\right),\nonumber
\end{eqnarray}
and thus
\begin{equation}
  \label{2.5}
  \|\sqrt{\varrho_0}\vartheta\|_\infty^2\leq\epsilon\left\|\frac{\vartheta_y}{\sqrt J}\right\|_2^2+C_\epsilon\left(1+\|\sqrt{\varrho_0}\vartheta\|_2^2+\int_0^t\|\sqrt{\varrho_0}\vartheta\|_\infty^2d\tau\right)
\end{equation}
for any $\epsilon>0$. Choosing $\epsilon$ sufficiently small and plugging (\ref{2.5}) into (\ref{2.4}) yield
\begin{align}
  \frac{d}{dt}\left(\|\sqrt{\varrho_0}E\|_2^2+\frac{A_1}{3\mu}\|\sqrt{\varrho_0}v^2\|_2^2\right)+A_1\left\|\frac{vv_y}{\sqrt J}\right\|_2^2+\frac{\kappa c_v}2\left\|\frac{\vartheta_y}{\sqrt J}\right\|_2^2\nonumber\\
  \leq C\left(1+\|\sqrt{\varrho_0}\vartheta\|_2^2+\int_0^t\|\sqrt{\varrho_0}\vartheta\|_\infty^2d\tau\right).\label{2.6}
\end{align}
Combining (\ref{2.5}) with (\ref{2.6}) leads to
\begin{eqnarray*}
  \frac{d}{dt}\left(\|\sqrt{\varrho_0}E\|_2^2+\frac{A_1}{3\mu}\|\sqrt{\varrho_0}v^2\|_2^2+\int_0^t
  \|\sqrt{\varrho_0}\vartheta\|_\infty^2d\tau\right)+A_1\left\|\frac{vv_y}{\sqrt J}\right\|_2^2+\frac\kappa2 \left\|\frac{\vartheta_y}{\sqrt J}\right\|_2^2 \\
  \leq C\left(1+\|\sqrt{\varrho_0}E\|_2^2+\int_0^t\|\sqrt{\varrho_0}\vartheta\|_\infty^2d\tau\right),
\end{eqnarray*}
which, together with the Gr\"onwall inequality, implies that
$$
\sup_{0\leq t\leq T}\|\sqrt{\varrho_0}E\|_2^2+\int_0^T\left(
  \|\sqrt{\varrho_0}\vartheta\|_\infty^2+\left\|\frac{vv_y}{\sqrt J}\right\|_2^2+  \left\|\frac{\vartheta_y}{\sqrt J}\right\|_2^2\right)dt\leq C.
$$
This completes the proof of the conclusion.
\end{proof}

\begin{corollary}
  \label{Cor2.1}
  There are two positive constants $\underline C$ and $\overline C$, such that
  $$
  \underline C\leq J\leq\overline C\quad\mbox{on }(\alpha, \beta)\times(0,T),\quad
  \int_0^T\|v_y\|_2^2dt\leq C.
  $$
\end{corollary}

\begin{proof}
The lower bound of $J$ follows directly from Proposition \ref{Prop2.1} while the upper bound of $J$ follows from
combining Proposition \ref{Prop2.1} and Proposition \ref{Prop2.2}. Testing (\ref{Eqv}) with $v$ and integrating by parts yield
\begin{eqnarray*}
  \frac12\frac{d}{dt}\|\sqrt{\varrho_0}v\|_2^2+\mu\left\|\frac{v_y}{\sqrt J}\right\|_2^2&=&R\int_{\alpha}^{\beta}\frac{\varrho_0}{J} \vartheta
  v_y dy \\
  \leq C\left\|\frac{v_y}{\sqrt J}\right\|_2\|\sqrt{\varrho_0}\vartheta\|_2&\leq&\frac\mu 2\left\|\frac{v_y}{\sqrt J}\right\|_2^2 +C \|\sqrt{\varrho_0}\vartheta\|_2^2,
\end{eqnarray*}
where the lower bound of $J$ was used,
and thus
\begin{equation*}
  \frac{d}{dt}\|\sqrt{\varrho_0}v\|_2^2 +\mu\left\|\frac{v_y}{\sqrt J}\right\|_2^2\leq C\|\sqrt{\varrho_0}\vartheta\|_2^2.
\end{equation*}
The second conclusion follows from this, the upper bound of $J$ just proved, and Proposition \ref{Prop2.2}.
\end{proof}

In the rest of this section, we always assume that $\beta-\alpha\geq1$. We will use the following elementary inequality.

\begin{lemma}
  \label{Lem2.1}
It holds that
\begin{equation*}
  \|f\|_{L^p((\alpha, \beta))}\leq C(\|f\|_{L^2((\alpha, \beta))}+\|f\|_{L^2((\alpha, \beta))}^{\frac12+\frac1p}\|f'\|_{L^2((\alpha, \beta))}^{\frac12-\frac1p}),
  \quad p\in[2,\infty],
\end{equation*}
for any $f\in H^1((\alpha, \beta))$, and for a positive constant $C$ depending only on $p$.
\end{lemma}

\begin{proof}
  This can be proved by scaling the corresponding inequality in $(\alpha,\beta)$ to that
  in $(0,1)$, applying the
  Gagliardo-Nirenberg inequality for functions in $H^1((0,1))$, and using the condition $\beta-\alpha\geq1$. Since the proof is straightforward, and thus is omitted here.
\end{proof}

Let $G$ be the effective viscous flux, i.e.,
\begin{equation*}
  G:=\mu\frac{v_y}{J}-\pi=\mu\frac{v_y}{J}-R\frac{\varrho_0\vartheta}{J}.\label{G}
\end{equation*}
Then, it holds that
\begin{equation}
  \label{2.8}
  G_t-\frac\mu J\left(\frac{G_y}{\varrho_0}\right)_y=-\frac{\kappa(\gamma-1)}{J}\left(\frac{\vartheta_y}{J}\right)_y-\gamma\frac{v_y}{J} G
\end{equation}
and
\begin{equation}
  G|_{y=\alpha, \beta}=0. \label{2.9}
\end{equation}

\begin{proposition}
  \label{Prop2.3}
  It holds that
  $$
  \sup_{0\leq t\leq T}\|G\|_2^2+\int_0^T\left(\left\|\frac{G_y}{\sqrt{\varrho_0}}\right\|_2^2+\|G\|_\infty^4\right)dt\leq
  C(1+\|G_0\|_2^2).
  $$
\end{proposition}

\begin{proof}
Testing (\ref{2.8}) with $JG$, using (\ref{EqJ}), (\ref{2.9}), Lemma \ref{Lem2.1},
Corollary \ref{Cor2.1}, and the Young inequality, one obtains
\begin{eqnarray*}
  &&\frac12\frac{d}{dt}\|\sqrt JG\|_2^2+\mu\left\|\frac{G_y}{\sqrt{\varrho_0}}\right\|_2^2\\
  &=&\kappa(\gamma-1)\int_{\alpha}^{\beta}\frac{\vartheta_y G_y}{J}dy+\left(\frac12-\gamma\right)
  \int_{\alpha}^{\beta} v_y G^2 dy\\
  &\leq&C\left(\|\vartheta_y\|_2\left\|\frac{G_y}{\sqrt{\varrho_0}}\right\|_2+\|v_y\|_2\|G\|_4^2\right)\\
  &\leq&C\left[\|\vartheta_y\|_2\left\|\frac{G_y}{\sqrt{\varrho_0}}\right\|_2+\|v_y\|_2\left(\|G\|_2^2+\|G\|_2^{\frac32}
  \|G_y\|_2^{\frac12}\right)\right]\\
  &\leq&\frac\mu2\left\|\frac{G_y}{\sqrt{\varrho_0}}\right\|_2^2+C[\|\vartheta_y\|_2^2+(1+\|v_y\|_2^2)\|G\|_2^2],
\end{eqnarray*}
that is,
$$
\frac{d}{dt}\|\sqrt JG\|_2^2+\mu\left\|\frac{G_y}{\sqrt{\varrho_0}}\right\|_2^2
\leq C[\|\vartheta_y\|_2^2+(1+\|v_y\|_2^2)\|G\|_2^2].
$$
Thanks to this and the Gr\"onwall inequality, the desired conclusion, except the estimate on
$\int_0^T\|G\|_\infty^4dt$, follows from  Proposition \ref{Prop2.2} and Corollary \ref{Cor2.1}. While the estimate for $\int_0^T\|G\|_\infty^4dt$ follows from Corollary \ref{Cor2.1}, Lemma \ref{Lem2.1}, and the estimate just proved.
\end{proof}

\begin{proposition}
  \label{Prop2.4}
  It holds that
  $$
  \sup_{0\leq t\leq T}\left(\left\|\frac{J_y}{\sqrt{\varrho_0}}\right\|_2^2+\|v_y\|_2^2+\|J_t\|_2^2\right)
  +\int_0^T\left(\|\sqrt{\varrho_0}v_t\|_2^2+\left\|\frac{v_{yy}}{\sqrt{\varrho_0}}\right\|_2^2\right)dt\leq C.
  $$
\end{proposition}

\begin{proof}
Note that $v_y=\frac1\mu(JG+R\varrho_0\vartheta)$ and $\sqrt{\varrho_0}v_t=\frac{G_y}{\sqrt{\varrho_0}}$. It follows from Proposition \ref{Prop2.2}, Proposition \ref{Prop2.3}, and Corollary \ref{Cor2.1} that
$$
\sup_{0\leq t\leq T}\|v_y\|_2^2+\int_0^T\|\sqrt{\varrho_0}v_t\|_2^2dt\leq C,
$$
which by (\ref{EqJ}) implies
$$
\sup_{0\leq t\leq T}\|J_t\|_2^2\leq C.
$$
Direct calculations yield
\begin{eqnarray*}
  J_{yt}=\frac1\mu(JG_y+J_yG+R\varrho_0'\vartheta+R\varrho_0\vartheta_y).
\end{eqnarray*}
Taking the inner product of the above with $\frac{J_y}{\varrho_0}$, one obtains from Proposition \ref{Prop2.2}, Corollary \ref{Cor2.1}, and (H1) that
\begin{eqnarray*}
  \frac12\frac{d}{dt}\left\|\frac{J_y}{\sqrt{\varrho_0}}\right\|_2^2&\leq&C\int_{\alpha}^{\beta}\left(|J|\left|\frac{G_y}{\sqrt{\varrho_0}}
  \right|+\left|\frac{J_y}{\sqrt{\varrho_0}}\right||G|+ \sqrt{\varrho_0}\vartheta+ \sqrt{\varrho_0}|\vartheta_y|\right)\frac{|J_y|}{\sqrt{\varrho_0}}
  dy \\
  &\leq&C\left(\left\|\frac{G_y}{\sqrt{\varrho_0}}\right\|_2+\|G\|_\infty\left\|\frac{J_y}{\sqrt{\varrho_0}}\right\|_2
  +\|\sqrt{\varrho_0}\vartheta\|_2+\|\vartheta_y\|_2\right)\left\|\frac{J_y}{\sqrt{\varrho_0}}\right\|_2 \\
  &\leq&C(1+\|G\|_\infty)\left\|\frac{J_y}{\sqrt{\varrho_0}}\right\|_2^2
  +C\left(1+\|\vartheta_y\|_2^2+\left\|\frac{G_y}{\sqrt{\varrho_0}}\right\|_2^2\right),
\end{eqnarray*}
which, together with the Gr\"onwall inequality, Proposition \ref{Prop2.2}, Corollary \ref{Cor2.1},
and Proposition \ref{Prop2.3}, yields
\begin{equation}
  \label{2.10}
  \sup_{0\leq t\leq T}\left\|\frac{J_y}{\sqrt{\varrho_0}}\right\|_2^2\leq C.
\end{equation}
Since
\begin{equation}
\label{VY-1}
v_{yy}=\frac1\mu(J_yG+JG_y+R\varrho_0'\vartheta+R\varrho_0\vartheta_y),
\end{equation}
it follows from (\ref{2.10}), Corollary \ref{Cor2.1}, Propositions \ref{Prop2.2}, Proposition \ref{Prop2.3}, and (H1) that
$$
\int_0^T\left\|\frac{v_{yy}}{\sqrt{\varrho_0}}\right\|_2^2dt\leq C\int_0^T\left(\left\|\frac{J_y}{\sqrt{\varrho_0}}\right\|_2^2
\|G\|_\infty^2+\left\|\left(\frac{G_y}{\sqrt{\varrho_0}},\sqrt{\varrho_0}\vartheta,\vartheta_y\right)\right\|_2^2\right)dt\leq C.
$$
This completes the proof.
\end{proof}

\begin{proposition}
  \label{Prop2.5}
  It holds that
  $$
  \sup_{0\leq t\leq T}\|\sqrt{\varrho_0}\vartheta_y\|_2^2+\int_0^T(\|\varrho_0\vartheta_t\|_2^2+\|\vartheta_{yy}\|_2^2)
  dt\leq C.
  $$
\end{proposition}

\begin{proof}
Rewrite (\ref{Eqtheta}) as
\begin{equation}
\label{2.11-1}
c_v\varrho_0\vartheta_t-\kappa\left(\frac{\vartheta_y}{J}\right)_y=v_yG.
\end{equation}
Note that $\vartheta_t|_{y=\alpha,\beta}=0$.
Taking the inner product of the above equation with $\varrho_0\vartheta_t$ yields
\begin{equation}
  \kappa\int_{\alpha}^{\beta}\frac{\vartheta_y}{J}(\varrho_0\vartheta_{yt}+\varrho_0'\vartheta_t)dy +c_v\|\varrho_0\vartheta_t\|_2^2 =
  \int_{\alpha}^{\beta}v_yG\varrho_0\vartheta_tdy. \label{2.11}
\end{equation}
It follows from (\ref{EqJ}) that
$$
\int_{\alpha}^{\beta}\frac{\vartheta_y}{J}\varrho_0\vartheta_{yt}dy=\frac12\frac{d}{dt}\left\|\sqrt{\frac{\varrho_0}{J}}\vartheta_y\right\|_2^2
+\frac12\int_{\alpha}^{\beta}\frac{v_y}{J^2}\varrho_0|\vartheta_y|^2dy.
$$
Substituting this into (\ref{2.11}) and using (H1) and Corollary \ref{Cor2.1}, one gets
\begin{eqnarray*}
  &&\frac\kappa2\frac{d}{dt}\left\|\sqrt{\frac{\varrho_0}{J}}\vartheta_y\right\|_2^2+c_v\|\varrho_0\vartheta_t\|_2^2\\
  &=&\int_{\alpha}^{\beta}\left(v_yG\varrho_0\vartheta_t-\frac\kappa2\frac{v_y}{J^2}\varrho_0|\vartheta_y|^2-\kappa\frac{\vartheta_y}{J}\varrho_0'\vartheta_t
  \right) dy\\
  &\leq&\int_{\alpha}^{\beta}\left(|v_y||G|\varrho_0|\vartheta_t|+\frac\kappa 2\frac{|v_y|}{J^2}\varrho_0|\vartheta_y|^2+\kappa K_1\frac{|\vartheta_y|}{J}
  \varrho_0|\vartheta_t|\right)dy\\
  &\leq&\frac{c_v}{2}\|\varrho_0\vartheta_t\|_2^2 +C(\|G\|_\infty^2\|v_y\|_2^2
  +\|v_y\|_\infty\|\sqrt{\varrho_0}\vartheta_y\|_2^2+\|\vartheta_y\|_2^2),
\end{eqnarray*}
which implies
\begin{eqnarray*}
  &&\kappa\frac{d}{dt}\left\|\sqrt{\frac{\varrho_0}{J}}\vartheta_y\right\|_2^2+c_v\|\varrho_0\vartheta_t\|_2^2\\
  &\leq&C\left[\|G\|_\infty^2\|v_y\|_2^2+(\|G\|_\infty+\|\varrho_0\vartheta\|_\infty)\|\sqrt{\varrho_0}\vartheta_y\|_2^2
  +\|\vartheta_y\|_2^2\right].
\end{eqnarray*}
It follows from this, the Gr\"onwall inequality, Propositions \ref{Prop2.2}--\ref{Prop2.4}, and Corollary \ref{Cor2.1} that
\begin{align}
  &\sup_{0\leq t\leq T}\|\sqrt{\varrho_0}\vartheta_y\|_2^2+\int_0^T\|\varrho_0\vartheta_t\|_2^2dt\nonumber\\
  \leq&~~~~ Ce^{C\int_0^T(\|G\|_\infty+\|\varrho_0\vartheta\|_\infty)dt}
  \left(\|\sqrt{\varrho_0}\vartheta_0'\|_2^2 +\int_0^T(\|G\|_\infty^2\|v_y\|_2^2+\|\vartheta_y\|_2^2)dt\right)
  \leq C. \label{2.12}
\end{align}
Direct calculations and using (\ref{2.11-1}) yield
\begin{equation*}
\kappa\vartheta_{yy}=\kappa\left(\frac{\vartheta_y}{J}\right)_yJ+\kappa\frac{\vartheta_y}{J}J_y=J(
c_v\varrho_0\vartheta_t-v_yG)+\kappa\frac{\vartheta_y}{J}J_y.
\end{equation*}
It follows from this, (\ref{2.12}), Propositions \ref{Prop2.3}--\ref{Prop2.4}, Corollary \ref{Cor2.1}, and
Lemma \ref{Lem2.1} that
\begin{eqnarray*}
  \int_0^T\|\vartheta_{yy}\|_2^2 dt&\leq&C\int_0^T\left(\|\varrho_0\vartheta_t\|_2^2+\|v_y\|_2^2\|G\|_\infty^2+\|\vartheta_y\|_\infty^2
  \|J_y\|_2^2\right)dt\\
  &\leq&C+C\int_0^T\|\vartheta_y\|_\infty^2dt\leq C+C\int_0^T\|\vartheta_y\|_2(\|\vartheta_y\|_2+\|\vartheta_{yy}\|_2)dt \\
  &\leq& \frac12\int_0^T\|\vartheta_{yy}\|_2^2dt+C,
\end{eqnarray*}
and thus $\int_0^T\|\vartheta_{yy}\|_2^2dt\leq C.$ This completes the proof.
\end{proof}

\begin{proposition}
  \label{Prop2.6}
It holds that
  $$
  \sup_{0\leq t\leq T}\left\|\frac{G_y}{\sqrt{\varrho_0}}\right\|_2^2+\int_0^T\left(\|G_t\|_2^2+\left\|\left(\frac{G_y}{\varrho_0}\right)_y
  \right\|_2^2\right)dt\leq C\left(1+\left\|\frac{G_0'}{\sqrt{\varrho_0}}\right\|_2^2\right).
  $$
\end{proposition}

\begin{proof}
Combining (\ref{2.8}) with (\ref{2.11-1}) yields
\begin{eqnarray*}
  G_t-\frac\mu J\left(\frac{G_y}{\varrho_0}\right)_y
%  &=&-\frac{\gamma-1}{J}\left(c_v\varrho_0\vartheta_t-v_yG\right)-\gamma
%  \frac{v_y}{J}G\\
  &=&-\frac RJ\varrho_0\vartheta_t-\frac{v_y}{J}G.
\end{eqnarray*}
Note that $G_t|_{y=\alpha,\beta}=0$.
Multiplying the above with $JG_t$, integrating by parts, and using Corollary \ref{Cor2.1} yield
\begin{eqnarray*}
   \frac\mu2\frac{d}{dt}\left\|\frac{G_y}{\sqrt{\varrho_0}}\right\|_2^2+\|\sqrt JG_t\|_2^2
  =-\int_{\alpha}^{\beta}\left(R \varrho_0\vartheta_t+  v_y G\right)G_t dy \\
   \leq \frac12\|\sqrt JG_t\|_2^2+C(\|\varrho_0\vartheta_t\|_2^2+\|v_y\|_2^2\|G\|_\infty^2),
\end{eqnarray*}
from which, by Propositions \ref{Prop2.3}--\ref{Prop2.5}, the conclusion follows.
\end{proof}

\begin{proposition}
  \label{Prop2.7}
  It holds that
  $$
  \sup_{0\leq t\leq T}\left(\left\|\varrho_0^{\frac32}\vartheta_t\right\|_2^2+\|\sqrt{\varrho_0}\vartheta_{yy}\|_2^2\right)
  +\int_0^T\|\varrho_0\vartheta_{yt}\|_2^2dt\leq C.
  $$
\end{proposition}

\begin{proof}
Note that $v_y=\frac1\mu(JG+R\varrho_0\vartheta)$ and
\begin{equation}
  v_{yt}=\frac1\mu(JG_t+v_yG+R\varrho_0\vartheta_t).\label{VY-2}
\end{equation}
It follows from (\ref{EqJ}), (\ref{2.11-1}), and direct calculations that
\begin{equation*}
  c_v\varrho_0\vartheta_{tt}-\kappa\left(\frac{\vartheta_{yt}}{J}\right)_y=
  -\kappa\left(\frac{v_y\vartheta_y}{J^2}\right)_y+\frac{v_y}{\mu}G^2+\frac1\mu(2JG+R\varrho_0\vartheta)G_t+\frac R\mu\varrho_0\vartheta_t G.
\end{equation*}
Note that $\vartheta_t|_{y=\alpha,\beta}=0$.
Multiplying the above equation with $\varrho_0^2\vartheta_t$ and integrating by parts yield
\begin{eqnarray*}
  &&\frac{c_v}{2}\frac{d}{dt}\left\|\varrho_0^{\frac32}\vartheta_t\right\|_2^2+\kappa\int_{\alpha}^{\beta}\frac{\vartheta_{yt}}{J}(\varrho_0^2\vartheta_{yt}
  +2\varrho_0\varrho_0'\vartheta_t)dy\\
  &=&\frac1\mu\int_{\alpha}^{\beta}
  [v_yG^2+(2JG+R\varrho_0\vartheta)G_t+R\varrho_0\vartheta_tG]\varrho_0^2\vartheta_tdy\\
  &&+\kappa\int_{\alpha}^{\beta}\frac{v_y\vartheta_y}{J^2}(\varrho_0^2\vartheta_{yt}+2\varrho_0\varrho_0'\vartheta_t)dy.
\end{eqnarray*}
Then, by Corollary \ref{Cor2.1} and (H1), one deduces
\begin{eqnarray*}
  &&\frac{c_v}{2}\frac{d}{dt}\|\varrho_0^{\frac32}\vartheta_t\|_2^2+\kappa\left\|\frac{\varrho_0}{\sqrt J}\vartheta_{yt}\right\|_2^2 \\
  &\leq&C\int_{\alpha}^{\beta}[\varrho_0^2|\vartheta_t||\vartheta_{yt}|+|v_y||\vartheta_y|(\varrho_0^2|\vartheta_{yt}|+\varrho_0^2|\vartheta_t|)]dy \\
  &&+C\int_{\alpha}^{\beta}[|v_y|G^2+(|G|+\varrho_0\vartheta)|G_t|+\varrho_0|\vartheta_t||G|]\varrho_0^2|\vartheta_t| dy \\
  &\leq&\frac\kappa2\left\|\frac{\varrho_0}{\sqrt J}\vartheta_{yt}\right\|_2^2+C(\|\varrho_0\vartheta_t\|_2^2+\|v_y\|_\infty^2\|\varrho_0\vartheta_y\|_2^2)
  +C\|G\|_\infty^2(\|v_y\|_2^2+\|\varrho_0^2\vartheta_t\|_2^2)\\
  &&+C\|G_t\|_2^2+C(\|G\|_\infty^2+\|\varrho_0\vartheta\|_\infty^2)\|\varrho_0^2\vartheta_t\|_2^2
  +C\|G\|_\infty\|\varrho_0^{\frac32}\vartheta_t\|_2^2,
\end{eqnarray*}
from which, by Propositions \ref{Prop2.4}--\ref{Prop2.5} and $v_y=\frac1\mu(JG+R\varrho_0\vartheta)$, one
obtains
\begin{eqnarray*}
  &&c_v\frac{d}{dt}\|\varrho_0^{\frac32}\vartheta_t\|_2^2+\kappa\left\|\frac{\varrho_0}{\sqrt J}\vartheta_{yt}\right\|_2^2\\
  &\leq&C(\|G\|_\infty^2+\|\varrho_0\vartheta\|_\infty^2+1)\|\varrho_0^{\frac32}\vartheta_t\|_2^2
  +C(\|v_y\|_\infty^2+\|G\|_\infty^2 +\|G_t\|_2^2+\|\varrho_0\vartheta_t\|_2^2)\\
  &\leq&C(\|G\|_\infty^2+\|\varrho_0\vartheta\|_\infty^2+1)\|\varrho_0^{\frac32}\vartheta_t\|_2^2+C(\|G\|_\infty^2+\|\varrho_0\vartheta\|_\infty^2
  +\|G_t\|_2^2+\|\varrho_0\vartheta_t\|_2^2).
\end{eqnarray*}
Applying the Gr\"onwall inequality to the above, one can get by Propositions \ref{Prop2.2}--\ref{Prop2.3} and \ref{Prop2.5}--\ref{Prop2.6}, and Corollary \ref{Cor2.1} that
\begin{eqnarray}
  &&\sup_{0\leq t\leq T}\|\varrho_0^{\frac32}\vartheta_t\|_2^2+\int_0^T\|\varrho_0\vartheta_{yt}\|_2^2dt\nonumber \\
  &\leq&Ce^{C\int_0^T(\|G\|_\infty^2+\|\varrho_0\vartheta\|_\infty^2)dt} \left.\big\|\varrho_0^{\frac32}
  \vartheta_t\big\|_2^2\right|_{t=0}
  \nonumber\\
  &&+Ce^{C\int_0^T(\|G\|_\infty^2+\|\varrho_0\vartheta\|_\infty^2)dt}
  \int_0^T(\|G\|_\infty^2+\|\varrho_0\vartheta\|_\infty^2+\|G_t\|_2^2+\|\varrho_0\vartheta_t\|_2^2)dt \nonumber \\
  &\leq&C(1+\|\sqrt{\varrho_0}\vartheta_0''\|_2^2+\|\sqrt{\varrho_0}v_0'G_0\|_2^2), \label{2.13-1}
\end{eqnarray}
where the fact that $\varrho_0^{\frac32}\vartheta_t|_{t=0}=\frac{\sqrt{\varrho_0}}{c_v}(\kappa\vartheta_0''+v_0'G_0)$ has been used, which
follows from (\ref{2.11-1}). Therefore, noticing that Lemma \ref{Lem2.1}
implies
$$
\|\sqrt{\varrho_0}v_0'G_0\|_2^2\leq C\|v_0'\|_\infty^2\|G_0\|_2^2\leq C\|v_0'\|_{H^1}^2\|G_0\|_2^2,
$$
one gets from (\ref{2.13-1}) that
\begin{equation}
\sup_{0\leq t\leq T}\|\varrho_0^{\frac32}\vartheta_t\|_2^2+\int_0^T\|\varrho_0\vartheta_{yt}\|_2^2dt
  \leq C(1+\|\sqrt{\varrho_0}\vartheta_0''\|_2^2).\label{2.13}
\end{equation}
Note that
\begin{equation*}
  \vartheta_{yy}=J\left(\frac{\vartheta_y}{J}\right)_y+\frac{\vartheta_y}{J}J_y
  =\frac1\kappa(c_v\varrho_0\vartheta_t-v_yG)J+\frac1J\vartheta_yJ_y.
\end{equation*}
It follows from this, (\ref{2.13}), Proposition \ref{Prop2.4}, and Corollary \ref{Cor2.1} that
\begin{eqnarray}
  \|\sqrt{\varrho_0}\vartheta_{yy}\|_2^2&\leq&C(\|\varrho_0^{\frac32}\vartheta_t\|_2^2+\|v_y\|_2^2\|G\|_\infty^2+\|\sqrt{\varrho_0}
  \vartheta_y\|_\infty^2\|J_y\|_2^2) \nonumber\\
  &\leq&C(1+\|G\|_\infty^2+\|\sqrt{\varrho_0}\vartheta_y\|_\infty^2).\label{2.14}
\end{eqnarray}
It remains to estimate $\|G\|_\infty^2$ and $\|\sqrt{\varrho_0}\vartheta_y\|_\infty^2$ as follows.
Note that Lemma \ref{Lem2.1}, Proposition \ref{Prop2.3}, and Proposition \ref{Prop2.6} imply that
\begin{equation}
  \label{2.15}
  \|G\|_\infty^2\leq C\|G\|_2(\|G\|_2+\|G_y\|_2)\leq C.
\end{equation}
By Lemma \ref{Lem2.1} and (H1), and Proposition \ref{Prop2.5}, it holds that
\begin{eqnarray}
  \|\sqrt{\varrho_0}\vartheta_y\|_\infty^2&\leq&C\|\sqrt{\varrho_0}\vartheta_y\|_2\left(\|\sqrt{\varrho_0}\vartheta_y\|_2+\|\sqrt{\varrho_0}\vartheta_{yy}\|_2
  +\left\|\frac{\varrho_0'}{\sqrt{\varrho_0}}\vartheta_y\right\|_2\right)\nonumber\\
  &\leq&C(1+\|\sqrt{\varrho_0}\vartheta_{yy}\|_2). \label{2.16}
\end{eqnarray}
Plugging (\ref{2.15}) and (\ref{2.16}) into (\ref{2.14}) and using the Cauchy inequality yield
\begin{equation*}
  \|\sqrt{\varrho_0}\vartheta_{yy}\|_2^2\leq C(1+\|\sqrt{\varrho_0}\vartheta_{yy}\|_2)\leq\frac{ \|\sqrt{\varrho_0}\vartheta_{yy}\|_2^2}{2}+C,
\end{equation*}
which gives $ \|\sqrt{\varrho_0}\vartheta_{yy}\|_2^2\leq C$. This completes the proof.
\end{proof}

\begin{proposition}
  \label{Prop2.8}
  It holds that
  $$
  \sup_{0\leq t\leq T}\left(\|\sqrt{\varrho_0}v_t\|_2^2+\left\|\frac{v_{yy}}{\sqrt{\varrho_0}}\right\|_2^2\right)
  +\int_0^T(\|v_{yt}\|_2^2+\|v_{yyy}\|_2^2+\|J_{yy}\|_2^2)dt\leq C.
  $$
\end{proposition}

\begin{proof}
The estimate for $\sqrt{\varrho_0}v_t$ follows directly from Proposition \ref{Prop2.6} since
$\sqrt{\varrho_0}v_t=\frac{G_y}{\sqrt{\varrho_0}}$.
It follows from (H1), (\ref{VY-1}), (\ref{VY-2}),
(\ref{2.15}), Corollary \ref{Cor2.1}, and Propositions \ref{Prop2.2}--\ref{Prop2.6} that
\begin{eqnarray*}
\left\|\frac{v_{yy}}{\sqrt{\varrho_0}}\right\|_2^2 \leq C\left(\left\|\frac{J_y}{\sqrt{\varrho_0}}\right\|_2^2\|G\|_\infty^2+\left\|\frac{G_y}{\sqrt{\varrho_0}}\right\|_2^2
+\|\sqrt{\varrho_0}\vartheta_y\|_2^2
+\|\sqrt{\varrho_0}\vartheta\|_2^2\right)\leq C,\\
\int_0^T\|v_{yt}\|_2^2dt \leq C\int_0^T(\|v_y\|_2^2\|G\|_\infty^2+\|G_t\|_2^2+\|\varrho_0\vartheta_t\|_2^2)dt\leq C.
\end{eqnarray*}
Noticing that
\begin{equation*}
  v_{yyy}=\frac1\mu(J_{yy}G+2J_yG_y+JG_{yy}+R\varrho_0''\vartheta+2R\varrho_0'\vartheta_y+R\varrho_0\vartheta_{yy}).
\end{equation*}
one can get from (H1), (\ref{2.15}), Corollary \ref{Cor2.1},
and Propositions \ref{Prop2.2}--\ref{Prop2.3} and \ref{Prop2.5}--\ref{Prop2.6} that
\begin{eqnarray}
  \int_0^t\|v_{yyy}\|_2^2d\tau&\leq&C\int_0^t(\|J_{yy}\|_2^2\|G\|_\infty^2+\|J_y\|_\infty^2\|G_y\|_2^2
  +\|G_{yy}\|_2^2\nonumber\\
  &&+\|\varrho_0\vartheta\|_2^2+\|\varrho_0\vartheta_y\|_2^2+\|\varrho_0\vartheta_{yy}\|_2^2)d\tau \nonumber\\
  &\leq&C\int_0^t(\|J_{yy}\|_2^2+\|J_y\|_\infty^2+\|G_{yy}\|_2^2)d\tau+C, \label{2.17}
\end{eqnarray}
where $\|G\|_\infty^2\leq C(\|G\|_2^2+\|G_y\|_2^2)$ guaranteed by Lemma \ref{Lem2.1} wa used.
Next, $\|J_y\|_\infty^2$ and $\|G_{yy}\|_2^2$ can be estimated as follows.
Lemma \ref{Lem2.1} and Proposition \ref{Prop2.4} imply that
\begin{equation}
  \label{2.18}
  \|J_y\|_\infty^2\leq C(\|J_y\|_2^2+\|J_y\|_2\|J_{yy}\|_2)\leq C(1+\|J_{yy}\|_2^2).
\end{equation}
While (H1) and Proposition \ref{Prop2.6} yield
\begin{eqnarray*}
  \int_0^T\|G_{yy}\|_2^2dt&\leq&\int_0^T\left(\left\|\varrho_0\left(\frac{G_y}{\varrho_0}\right)_y\right\|_2
  +\left\|\varrho_0'\frac{G_y}{\varrho_0}\right\|_2\right)^2dt\\
  &\leq& C\int_0^T\left(\left\|\left(\frac{G_y}{\varrho_0}\right)_y\right\|_2^2+\left\|G_y\right\|_2^2\right)dt
  \leq C.
\end{eqnarray*}
It follows from this, (\ref{2.17}), and (\ref{2.18}) that
\begin{equation}
  \int_0^t\|v_{yyy}\|_2^2d\tau\leq C\left(1+\int_0^t\|J_{yy}\|_2^2d\tau\right). \label{2.19}
\end{equation}
Since $J_{yy}=\int_0^tv_{yyy}d\tau$, one has
\begin{equation}
  \int_0^t\|J_{yy}\|_2^2d\tau\leq \int_0^t\left\|\int_0^\tau v_{yyy}d\tau'\right\|_2^2d\tau
  \leq C\int_0^t\left(\int_0^\tau\|v_{yyy}\|_2^2d\tau'\right)d\tau.\label{2.19-1}
\end{equation}
Plugging this into (\ref{2.19}) leads to
\begin{equation*}
  \int_0^t\|v_{yyy}\|_2^2d\tau\leq C+C\int_0^t\left(\int_0^\tau\|v_{yyy}\|_2^2d\tau'\right)d\tau,
\end{equation*}
which implies
$\int_0^T\|v_{yyy}\|_2^2dt\leq Ce^T\leq C$ by the Gr\"onwall inequality.
This, together with (\ref{2.19-1}), shows that
$\int_0^t\|J_{yy}\|_2^2d\tau\leq C.$
This completes the proof.
\end{proof}

\begin{proposition}
  \label{Prop2.9}
  It holds that
  $$
  \sup_{0\leq t\leq T}(\|J_{yy}\|_2^2+\|J_{yt}\|_2^2)\leq C.
  $$
\end{proposition}

\begin{proof}
This follows from Proposition \ref{Prop2.8} by using $J_{yy}=\int_0^tv_{yyy}d\tau$ and $J_{yt}=v_{yy}$.
\end{proof}

\begin{proposition}
  \label{Prop2.10}
  It holds that
  $$
  \int_0^T\|\vartheta_{yyy}\|_2^2dt\leq C.
  $$
\end{proposition}

\begin{proof}
By Lemma \ref{Lem2.1} and Propositions \ref{Prop2.3}, \ref{Prop2.4}, \ref{Prop2.6}, \ref{Prop2.8}, and \ref{Prop2.9}, one has
\begin{equation}
  \|J_y\|_\infty+\|J_{yy}\|_2+\|v_y\|_\infty+\|G\|_\infty\leq C. \label{2.20-1}
\end{equation}
It follows from (\ref{2.11-1})
that
%\begin{equation*}
%  c_v\varrho_0\vartheta_t-\kappa\frac{\vartheta_{yy}}{J}+\kappa\frac{\vartheta_yJ_y}{J^2}=v_yG,
%\end{equation*}
%which gives
%\begin{equation*}
%  \vartheta_{yy}=\frac{c_v}{\kappa}\varrho_0J\vartheta_t+\frac{\vartheta_yJ_y}{J}-\frac1\kappa Jv_yG.
%\end{equation*}
%Therefore,
\begin{eqnarray*}
  \vartheta_{yyy}&=&\frac{c_v}{\kappa}(\varrho_0'J\vartheta_t+2\varrho_0 J_y\vartheta_t+\varrho_0J\vartheta_{yt})+\frac{\vartheta_y}J J_{yy}\\
  &&-\frac1\kappa(2J_yv_yG+Jv_{yy}G+Jv_yG_y).
\end{eqnarray*}
Then, by Corollary \ref{Cor2.1}, (H1), (\ref{2.20-1}), and Proposition \ref{Prop2.9}, one deduces
\begin{eqnarray*}
  \int_0^T\|\vartheta_{yyy}\|_2^2dt
  &\leq&C\int_0^T(\|\varrho_0\vartheta_t\|_2^2+\|J_y\|_\infty^2\|\varrho_0\vartheta_t\|_2^2
  +\|\varrho_0\vartheta_{yt}\|_2^2 +\|\vartheta_y\|_\infty^2\|J_{yy}\|_2^2 \\
  &&+\|J_y\|_\infty^2\|v_y\|_2^2\|G\|_\infty^2+\|v_{yy}\|_2^2\|G\|_\infty^2
  +\|v_y\|_\infty^2\|G_y\|_2^2)dt \\
  &\leq&C\int_0^T(\|\varrho_0\vartheta_t\|_2^2
  +\|\varrho_0\vartheta_{yt}\|_2^2+\|\vartheta_y\|_2^2 \\
  &&+\|\vartheta_{yy}\|_2^2 + \|v_y\|_2^2 +\|v_{yy}\|_2^2
  + \|G_y\|_2^2)dt,
\end{eqnarray*}
where $\|\vartheta_y\|_\infty^2\leq C(\|\vartheta_y\|_2^2+\|\vartheta_{yy}\|_2^2)$ guaranteed by Lemma \ref{Lem2.1} was used,
from which, by Corollary \ref{Cor2.1} and Propositions \ref{Prop2.2}--\ref{Prop2.8}, it follows $\int_0^T\|\vartheta_{yyy}\|_2^2dt\leq C$. This proves the conclusion.
\end{proof}

As a consequence of Propositions \ref{BASIC}--\ref{Prop2.10} and Corollary \ref{Cor2.1}, one has:

\begin{corollary}
\label{CorAPRI}
Let $(J, v, \vartheta)$ be the unique global solution stated in Proposition \ref{PROPGLOBAL} to system (\ref{EqJ})--(\ref{Eqtheta}), subject to (\ref{IC})--(\ref{BC}), and $\mathscr N_0$ be given by (\ref{N0}). Then, for any $T\in[0,\infty)$, it holds that
\begin{align*}
\inf_{(\alpha,\beta)\times(0,T)}J\geq\underline C_T, \quad\sup_{0\leq t\leq T} \left\|\left(\frac{J_y}{\sqrt{\varrho_0}}, J_{yy}, J_t, J_{yt}\right)\right\|_{L^2((\alpha,\beta))}^2
&\leq C_T, \\
\sup_{0\leq t\leq T} \left\|\left(\sqrt{\varrho_0}v,\sqrt{\varrho_0}v^2, v_y, \frac{v_{yy}}{\sqrt{\varrho_0}},
\sqrt{\varrho_0}v_t\right)\right\|_{L^2((\alpha,\beta))}^2& \\
+\int_0^T \|(v_{yyy},v_{yt}) \|_{L^2((\alpha,\beta))}^2 dt&\leq C_T,\\
\sup_{0\leq t\leq T}\left(\|\varrho_0\vartheta\|_{L^1((\alpha,\beta))}+ \left\|\left(\sqrt{\varrho_0}\vartheta, \sqrt{\varrho_0}\vartheta_y, \sqrt{\varrho_0}\vartheta_{yy}, \varrho_0^\frac32\vartheta_t\right)\right\|_{L^2((\alpha,\beta))}^2\right)&  \\
 +\int_0^T\left(\|\vartheta_y\|_{H^2((\alpha,\beta))}^2+\|(\varrho_0\vartheta_t,\varrho_0\vartheta_{yt})
  \|_{L^2((\alpha,\beta))}^2\right)dt& \leq C_T,\\
\sup_{0\leq t\leq T} \left\|\frac{G_y}{\sqrt{\varrho_0}}\right\|_{L^2((\alpha,\beta))}^2
+\int_0^T\left(\left\|\left(\frac{G_y}{\varrho_0}\right)_y\right\|_{L^2((\alpha,\beta))}^2+
  \left\|G_t\right\|_{L^2((\alpha,\beta))}^2\right)dt& \leq C_T,
\end{align*}
where $\underline C_T$ and $C_T$ are positive constants depending only on $R, c_v, \mu, \kappa,
K_1, T$, and the upper bound of $\mathscr N_0$, but independent of $\alpha$ and $\beta$ with $\beta-\alpha\geq1$.
\end{corollary}

\section{Global existence of solutions: proof of Theorem \ref{THMGLOBAL}}
\label{SECGLOBAL}
\begin{proof}[Proof of Theorem \ref{THMGLOBAL}] The proof is given in three steps as follows.

\textbf{Step 1. Approximations of the initial data.} By the
assumption (\ref{H3}), there are two sequences $\{\alpha_n\}_{n=1}^\infty$ and $\{\beta_n\}_{n=1}^\infty$, with
$\lim_{n\rightarrow\infty}\alpha_n=-\infty$ and $\lim_{n\rightarrow\infty}\beta_n=\infty$, and a positive constant $M_0$,
such that
\begin{equation}
  \label{AS2}
\left|\frac{v_0'(\alpha_n)}{\sqrt{\varrho_0(\alpha_n)}}\right|+\left|\frac{v_0'(\beta_n)}{\sqrt{\varrho_0(\beta_n)}}\right|\leq M_0,\quad\forall n\geq1.
\end{equation}

Set $I_n=(\alpha_n-1, \beta_n+1)$. For each $n$, choose $0\leq\chi_n\in C_0^\infty(I_n)$,
satisfying
\begin{equation}
  \chi\equiv1\mbox{ on }[\alpha_n, \beta_n], \quad 0\leq\chi_n\leq1\mbox{ and } |\chi_n'|+|\chi_n''|\leq C_0 \mbox{ on }I_n, \label{GB-1}
\end{equation}
for a positive constant $C_0$ independent of $n$. Define $v_{0n}$ and $\vartheta_{0n}$ as
\begin{equation*}
  \vartheta_{0n} =\vartheta_0 \chi_n,
\end{equation*}
and
\begin{equation*}
  v_{0n} =\left\{
  \begin{array}{ll}
  v_0(\alpha_n)+ \frac2\pi v_0'(\alpha_n)\sin\left(\frac\pi2(y-\alpha_n)\right),&y\in[\alpha_n-1,\alpha_n],\\
  v_0(y),&y\in[\alpha_n,\beta_n],\\
  v_0(\beta_n)+ \frac2\pi v_0'(\beta_n)\sin\left(\frac\pi2(y-\beta_n)\right),&y\in[\beta_n,\beta_n+1].
  \end{array}
  \right.
\end{equation*}

It can be checked easily that
\begin{equation}
  v_{0n}'(\alpha_n-1)=v_{0n}'(\beta_n+1)=\vartheta_{0n}(\alpha_n-1)=\vartheta_{0n}(\beta_n+1)=0. \label{GB-2}
\end{equation}
Noticing that
$$
v_{0n}(\alpha_n)=v_0(\alpha_n), \quad v_{0n}'(\alpha_n)=v_0'(\alpha_n),\quad v_{0n}(\beta_n)=v_0(\beta_n),\quad
  v_{0n}'(\beta_n)=v_0'(\beta_n),
$$
and since $v_0\in H_{loc}^2(\mathbb R)$ and $0\leq\vartheta_0\in H^2_{loc}(\mathbb R)$, one has
\begin{equation}
  v_{0n}\in H^2(I_n),\quad
  0\leq\vartheta_{0n}\in H^2(I_n).\label{GB-3}
\end{equation}

Due to $0\leq\chi_n\leq1$, it is clear that
\begin{equation}
  \|\sqrt{\varrho_0}\vartheta_{0n}\|_{L^2(I_n)}\leq\|\sqrt{\varrho_0}\vartheta_0\|_2. \label{GB-4}
\end{equation}
For any $y\in[\alpha_n-1, \alpha_n)$, the definition of $v_{0n}$ implies that
\begin{eqnarray*}
  |v_{0n}(y)-v_0(y)| \leq |v_0(\alpha_n)-v_0(y)|+\frac2\pi|v_0'(\alpha_n)|
%%  \\&\leq&|v_0'(\xi_n)||y-\alpha_n|+\frac2\pi|v_0'(\alpha_n)|
  \leq 2\|v_0'\|_\infty.
\end{eqnarray*}
%%where $\xi_n\in(y, \alpha_n)$.
Similarly, it holds that $|v_{0n}(y)-v_0(y)|\leq 2\|v_0'\|_\infty$, for any $y\in(\beta_n, \beta_n+1]$. As a result, one has
\begin{equation}
  |v_{0n}(y)-v_0(y)|\leq 2\|v_0'\|_\infty,\quad\forall y\in I_n. \label{GB-5}
\end{equation}
Hence
\begin{eqnarray}
  \|\sqrt{\varrho_0}v_{0n}\|_{L^2(I_n)}&\leq&\|\sqrt{\varrho_0}(v_{0n}-v_0)\|_{L^2(I_n)}+\|\sqrt{\varrho_0}
  v_0\|_{L^2(I_n)}\nonumber\\
  %%&\leq&2\|v_0'\|_\infty\|\sqrt{\varrho_0}\|_2+\|\sqrt{\varrho_0}v_0\|_2\nonumber \\
  &=&2\|v_0'\|_\infty\|\varrho_0\|_1^\frac12+\|\sqrt{\varrho_0}v_0\|_2, \label{GB-6}
\end{eqnarray}
and
\begin{eqnarray}
  \|\sqrt{\varrho_0}|v_{0n}|^2\|_{L^2(I_n)}&\leq&2\left\|\sqrt{\varrho_0}(|v_0|^2+|v_0-v_{0n}|^2)\right\|_{L^2(I_n)}\nonumber\\
 %% &\leq&2\|\sqrt{\varrho_0}|v_0|^2\|_2+8\|v_0'\|_\infty^2\|\sqrt{\varrho_0}\|_2 \nonumber\\
  &\leq&2\|\sqrt{\varrho_0}|v_0|^2\|_2+8\|v_0'\|_\infty^2\| \varrho_0 \|_1^\frac12. \label{GB-7}
\end{eqnarray}

It follows from (\ref{GB-1}) and direct calculations that
\begin{eqnarray}
  \|\sqrt{\varrho_0}\vartheta_{0n}'\|_{L^2(I_n)}&\leq&\|\sqrt{\varrho_0}\vartheta_0'\|_2+C_0\|\sqrt{\varrho_0}\vartheta_0\|_2,
  \label{GB-8}\\
  \|\sqrt{\varrho_0}\vartheta_{0n}''\|_{L^2(I_n)}&\leq&\|\sqrt{\varrho_0}\vartheta_0''\|_2+2C_0
  (\|\sqrt{\varrho_0}\vartheta_0'\|_2+\|\sqrt{\varrho_0}\vartheta_0\|_2).
  \label{GB-9}
\end{eqnarray}
By direct calculations,
%Since
%\begin{eqnarray*}
%  v_{0n}'(y)&=&\left\{
%  \begin{array}{ll}
%  v_0'(\alpha_n)\cos\left(\frac\pi2(y-\alpha_n)\right),&y\in(\alpha_n-1,\alpha_n),\\
%  v_0'(y),&y\in (\alpha_n,\beta_n),\\
%  v_0'(\beta_n)\cos\left(\frac\pi2(y-\beta_n)\right),&y\in(\beta_n,\beta_n+1),
%  \end{array}
%  \right.\\
%  v_{0n}''(y)&=&\left\{
%  \begin{array}{ll}
%  -\frac\pi2 v_0'(\alpha_n)\sin\left(\frac\pi2(y-\alpha_n)\right),&y\in(\alpha_n-1,\alpha_n),\\
%  v_0''(y),&y\in(\alpha_n,\beta_n),\\
%  -\frac\pi2v_0'(\beta_n)\sin\left(\frac\pi2(y-\beta_n)\right),&y\in(\beta_n,\beta_n+1),
%  \end{array}
%  \right.
%\end{eqnarray*}
one gets by the Sobolev inequality that
\begin{eqnarray}
  \|v_{0n}'\|_{H^1(I_n)}&\leq&\|v_0'\|_{H^1}+C(|v_0'(\alpha_n)|+|v_0'(\beta_n)|)\nonumber\\
  &\leq& \|v_0'\|_{H^1((\alpha_n, \beta_n))}+\|v_0'\|_{H^1((\alpha_n-1,\alpha_n)\cup(\beta_n, \beta_n+1))}\nonumber\\
  &\leq& C\|v_0'\|_{H^1}^2, \label{GB-10}
\end{eqnarray}
for a positive constant $C$ independent of $n$.

Set $G_{0n}=\mu v_{0n}'-R\varrho_0\vartheta_{0n}.$
Combining (\ref{GB-4}) with (\ref{GB-10}) leads to
\begin{eqnarray}
  \|G_{0n}\|_{L^2(I_n)}
  &\leq&\mu\|v_{0n}'\|_{L^2(I_n)}+R\|\varrho_0\|_\infty^\frac12\|\sqrt{\varrho_0}\vartheta_{0n}\|_{L^2(I_n)}\nonumber\\
  &\leq& C(\|v_{0}'\|_{H^1}+\|\varrho_0\|_\infty^\frac12\|\sqrt{\varrho_0}\vartheta_0\|_2),\label{GB-11}
\end{eqnarray}
for a positive constant $C$ independent of $n$.

For $y\in(\beta_n, \beta_n+1)$, one has
\begin{eqnarray}
  \frac{\varrho_0(\beta_n)}{\varrho_0(y)}
  %&=&1+\int_{\beta_n}^y\left(\frac{\varrho_0(\beta_n)}{\varrho_0(z)}\right)'dz
%  =1-\int_{\beta_n}^y \frac{\varrho_0(\beta_n)\varrho_0'(z)}{\varrho_0^2(z)} dz\nonumber\\
  &=& 1+ \int_{\beta_n}^y k(z)\frac{\varrho_0(\beta_n)}{\varrho_0(z)}dz, \quad\mbox{where }k(z)=-\frac{\varrho_0'(z)}{\varrho_0(z)}. \label{GB-11'}
\end{eqnarray}
By (H1), it holds that $|k(z)|\leq K_1$, for any $z\in\mathbb R$. Set
$$
f(y)=1+\int_{\beta_n}^y k(z)\frac{\varrho_0(\beta_n)}{\varrho_0(z)}dz, \quad\forall y\in(\beta_n, \beta_n+1).
$$
Then, it follows from (\ref{GB-11'}) that
\begin{equation*}
  f'(y)=k(y)\frac{\varrho_0(\beta_n)}{\varrho_0(y)}=k(y)f(y),
\end{equation*}
and thus
$$
f(y)=e^{\int_{\beta_n}^yk(z)dz}f(\beta_n)=e^{\int_{\beta_n}^yk(z)dz}\leq e^{K_1}, \quad\forall y\in(\beta_n, \beta_n+1).
$$
It follows from this and (\ref{GB-11'}) that
\begin{equation}
  \frac{\varrho_0(\beta_n)}{\varrho_0(y)}= f(y)\leq e^{K_1},\quad \forall y\in(\beta_n, \beta_n+1). \label{GB-12}
\end{equation}
Similarly, one has
\begin{equation}
  \frac{\varrho_0(\alpha_n)}{\varrho_0(y)} \leq e^{K_1}, \quad \forall y\in(\alpha_n-1, \alpha_n). \label{GB-13}
\end{equation}

Recall that $G_{0n}=\mu v_{0n}'-R\varrho_0\vartheta_{0n}.$ Then, direct calculations yield
\begin{equation*}
\frac{G_{0n}'}{\sqrt{\varrho_0}}
=\left\{
\begin{array}{ll}
-\frac\pi2\mu\frac{v_0'(\alpha_n)}{\sqrt{\varrho_0}}\sin\left(\frac\pi2(y-\alpha_n)\right)-R\left(\sqrt{\varrho_0}
\vartheta_{0n}'+\frac{\varrho_0'}{\sqrt{\varrho_0}}\vartheta_{0n}\right),& y\in(\alpha_n-1,\alpha_n),\\
\frac{G_0'}{\sqrt{\varrho_0}}, &y\in(\alpha_n, \beta_n), \\
-\frac\pi2\mu\frac{v_0'(\beta_n)}{\sqrt{\varrho_0}}\sin\left(\frac\pi2(y-\beta_n)\right)-R\left(\sqrt{\varrho_0}
\vartheta_{0n}'+\frac{\varrho_0'}{\sqrt{\varrho_0}}\vartheta_{0n}\right),& y\in(\beta_n,\beta_n+1).
\end{array}
\right.
\end{equation*}
It follows from (\ref{AS2}) and (\ref{GB-12})--(\ref{GB-13}) that
\begin{eqnarray*}
\left|\frac{v_0'(\alpha_n)}{\sqrt{\varrho_0(y)}}\right|+\left|\frac{v_0'(\beta_n)}{\sqrt{\varrho_0(y)}}\right|
&=&\left|\frac{v_0'(\alpha_n)}{\sqrt{\varrho_0(\alpha_n)}}\sqrt{\frac{\varrho_0(\alpha_n)}
{{\varrho_0(y)}}}\right|+\left|\frac{v_0'(\beta_n)}{\sqrt{\varrho_0(\beta_n)}}\sqrt{\frac{\varrho_0(\beta_n)}
{{\varrho_0(y)}}}\right|\\
&\leq& 2M_0e^{\frac{K_1}{2}}, \quad\forall y\in(\alpha_n-1, \alpha_n)\cup(\beta_n, \beta_n+1).
\end{eqnarray*}
This together with (H1) yields
\begin{equation*}
  \left|\frac{G_{0n}'(y)}{\sqrt{\varrho_0(y)}}\right|
  \leq
   \pi \mu M_0 e^{\frac{K_1}{2}}+R\left(\sqrt{\varrho_0}|\vartheta_{0n}'|+K_1\sqrt{\varrho_0}\vartheta_{0n}\right),
\end{equation*}
for any $y\in(\alpha_n-1, \alpha_n)\cup(\beta_n, \beta_n+1)$. Due to this and that
$\frac{G_{0n}'}{\sqrt{\varrho_0}}=\frac{G_{0}'}{\sqrt{\varrho_0}}$
on $(\alpha_n, \beta_n)$, it follows from (\ref{GB-8}) that
\begin{eqnarray}
  \left\|\frac{G_{0n}'}{\sqrt{\varrho_0}}\right\|_{L^2(I_n)}&\leq&\left\|\frac{G_{0}'}{\sqrt{\varrho_0}}\right\|_2
  +2\mu\pi M_0 e^{K_1/2}+R\left(\|\sqrt{\varrho_0}\vartheta_{0n}'\|_{L^2(I_n)}+K_1\|\sqrt{\varrho_0}\vartheta_{0n}\|_{L^2(I_n)}\right)
  \nonumber\\
  &\leq&\left\|\frac{G_{0}'}{\sqrt{\varrho_0}}\right\|_2
  + C\left(
  \|\sqrt{\varrho_0}\vartheta_{0}'\|_2+\|\sqrt{\varrho_0}\vartheta_{0}\|_2+1\right),\label{GB-14}
\end{eqnarray}
for a positive constant $C$ independent of $n$.

\textbf{Step 2. Solutions to the system in $I_n\times(0,\infty)$ and a priori estimates.}

For each positive integer $n$, let $(v_{0n}, \vartheta_{0n})$ be the initial data constructed as before.
Consider the initial-boundary value problem to the system (\ref{EqJ})--(\ref{Eqtheta})
in $(\alpha_n-1, \beta_n+1)\times(0,\infty)$, subject to
\begin{align}
  (J, v, \vartheta)|_{t=0}=(1, v_{0n}, \vartheta_{0n}),\quad
  (v_y, \vartheta)|_{y=a_n-1, \beta_n+1}=(0, 0).\label{IBC'}
\end{align}
Thanks to (\ref{GB-2}) and (\ref{GB-3}), and noticing that $\inf_{y\in I_n}\varrho_0>0$,
one can verify that the initial datum $(v_{0n}, \vartheta_{0n})$ satisfies all the assumptions
in Proposition \ref{PROPGLOBAL}, for each fixed $n$. Thus, there is a unique global strong solution $(J_n, v_n, \vartheta_n)$ to (\ref{EqJ})--(\ref{Eqtheta}) with (\ref{IBC'}).
Moreover, due to (\ref{GB-4}), (\ref{GB-6})--(\ref{GB-11}), and (\ref{GB-14}), it follows from Corollary \ref{CorAPRI}
that
\begin{align}
&\inf_{I_n\times(0,T)}J_n\geq\underline C_T, \quad\vartheta_n(y,t)\geq0,\label{AP1}\\
&\sup_{0\leq t\leq T} \left\|\left(\frac{\partial_yJ_n}{\sqrt{\varrho_0}}, \partial_y^2J_n, \partial_tJ_n, \partial_{yt}J_n\right)\right\|_{L^2(I_n)}^2 \leq C_T,  \\
&\sup_{0\leq t\leq T} \left\|\left(\sqrt{\varrho_0}v_n,\sqrt{\varrho_0}v_n^2, \partial_yv_n, \frac{\partial_y^2v_n}{\sqrt{\varrho_0}},
\sqrt{\varrho_0}\partial_tv_n\right)\right\|_{L^2(I_n)}^2\nonumber\\
&\quad\quad\quad+\int_0^T\|(\partial_y^3v_n, \partial_{yt}^2v_n)\|_{L^2(I_n)}^2 dt\leq C_T,
 \label{AP3}\\
&\sup_{0\leq t\leq T}\left(\|\varrho_0\vartheta_n\|_{L^1(I_n)}+ \left\|\left(\sqrt{\varrho_0}\vartheta_n, \sqrt{\varrho_0}\partial_y\vartheta_n, \sqrt{\varrho_0}\partial_y^2\vartheta_n, \varrho_0^\frac32\partial_t\vartheta_n\right)\right\|_{L^2(I_n)}^2\right)\nonumber\\
&\quad\quad\quad+\int_0^T\left(\|\partial_y\vartheta_n\|_{H^2(I_n)}^2+\|(\varrho_0\partial_t\vartheta_n,
  \varrho_0\partial_{yt}\vartheta_n)\|_{L^2(I_n)}^2\right)dt \leq C_T, \label{AP4}
\end{align}
and
\begin{align}
&  \sup_{0\leq t\leq T} \left\|\frac{\partial_yG_n}{\sqrt{\varrho_0}}\right\|_{L^2(I_n)}^2
+\int_0^T\left(\left\|\left(\frac{\partial_yG_n}{\varrho_0}
\right)_y\right\|_{L^2(I_n)}^2+
  \left\|\partial_tG_n\right\|_{L^2(I_n)}^2\right)dt\leq C_T, \label{AP5}
\end{align}
for any positive time $T$, where $G_n=\mu\frac{\partial_yv_n}{J_n}-R\frac{\varrho_0}{J_n}\vartheta_n$, and
$\underline C_T$ and $C_T$ are positive constants independent of $n$.

\textbf{Step 3. Convergence and existence. }

Thanks to the a priori estimates (\ref{AP1})--(\ref{AP5}) and $\inf_{(-k, k)}\varrho_0(y)>0$ for any $k\in\mathbb N$, the following estimate holds
\begin{align*}
  \|(J_n, v_n, \vartheta_n)&\|_{L^\infty(0,T; H^2((-k, k)))}+ \|(v_n, \vartheta_n)\|_{L^2(0,T; H^3((-k, k)))}+\|\partial_tJ_n\|_{L^\infty(0,T;
  H^1((-k,k)))}\\
  &+\|(\partial_tv_n, \partial_t\vartheta_n)\|_{L^\infty(0,T; L^2((-k, k)))\cap L^2(0,T; H^1((-k, k)))}\leq C_{k,T},\quad\forall k\in\mathbb N,
\end{align*}
for a positive constant $C_{k,T}$ independent of $n$. Due to this and the Cantor's diagonal
argument, there is a subsequence, still denoted by $(J_n, v_n, \vartheta_n)$, and $(J, v, \vartheta)$, such that
\begin{align}
&  (J_n, v_n, \vartheta_n)\overset{*}{\rightharpoonup}(J, v, \vartheta),\quad\mbox{ in }L^\infty(0,T; H^2((-k, k))),
\label{COV-1}\\
&  (v_n, \vartheta_n) \rightharpoonup (v, \vartheta),\quad\mbox{ in }L^2(0,T; H^3((-k, k))),\label{COV-2}\\
&  \partial_tJ_n\overset{*}{\rightharpoonup}J_t,\quad \mbox{ in }L^\infty(0,T;  H^1((-k,k))), \label{COV-3}\\
&  (\partial_tv_n, \partial_t\vartheta_n)\overset{*}{\rightharpoonup}(v_t, \vartheta_t),\quad\mbox{ in }L^\infty(0,T; L^2((-k, k))),\label{COV-4}\\
&  (\partial_tv_n, \partial_t\vartheta_n) \rightharpoonup (v_t, \vartheta_t),\quad\mbox{ in }L^2(0,T; H^1((-k, k))),\label{COV-5}
\end{align}
for any $k\in\mathbb N$.
Moreover, since $H^3((-k, k))\hookrightarrow\hookrightarrow C^2([-k,k])$ and $H^2((-k,k))\hookrightarrow\hookrightarrow C^1([-k,k])$, it follows
from the Aubin--Lions lemma that
\begin{align}
&  (J_n, v_n, \vartheta_n) \rightarrow (J, v, \vartheta),\quad\mbox{ in }C([0,T]; C^1([-k, k])),\label{COV-6}\\
&  (v_n, \vartheta_n) \rightarrow (v, \vartheta),\quad\mbox{ in }L^2(0,T; C^2([-k, k])),\label{COV-7}
\end{align}
for any $k\in\mathbb N$.
Thanks to these and by (\ref{AP1}), one has
\begin{equation}
  \inf_{(y,t)\in\mathbb R\times(0,T)}J(y,t)\geq\underline C_T, \quad \frac1{J_n}\rightarrow\frac1J\mbox{ in }C([0,T]; C^1([-k,k])),
  \label{COV-8}
\end{equation}
for any $k\in\mathbb N$.

Thanks to (\ref{COV-1})--(\ref{COV-8}) and noticing that $(v_{0n}, \vartheta_{0n})\rightarrow(v_0, \vartheta_0)$
in $H^2((-L, L))$ for any $L>0$,
one can take the limit as $n\rightarrow\infty$ to show that $(J, v, \vartheta)$ is a solution to the Cauchy
problem to the system (\ref{EqJ})--(\ref{Eqtheta}) subject to $(J, v, \vartheta)|_{t=0}=(1, v_0, \vartheta_0)$. The desired
regularities of $(J, v, \vartheta)$ stated in Theorem \ref{THMGLOBAL} follow
from the a priori estimates (\ref{AP1})--(\ref{AP5}) and convergence (\ref{COV-1})--(\ref{COV-7}) by the weakly lower
semi-continuity of norms. This proves Theorem \ref{THMGLOBAL}.
\end{proof}

\section{A Hopf type lemma and unboundedness of the entropy}\label{SECUNBDDENTROPY}

In this section, we prove the unboundedness of the entropy immediately after the initial time, i.e.\,Theorem \ref{UNBDDENTROPY}. As stated in
the Introduction, this is
based on some suitable scaling transform and a Hopf type lemma for a class of general linear degenerate equations.
So, we first establish a Hopf type lemma in the first subsection and then
present the proof of Theorem \ref{UNBDDENTROPY} in the second subsection.
The Hopf type lemma has its own independent interests and will also be applied to prove the uniform positivity of the
temperature in the next section.

\subsection{A Hopf type lemma}
Since the results in this subsection hold in any dimension, we use the following notations. Denote by $x=(x_1, x_2, \cdots, x_n)$ and $t$ the spatial and time variables respectively and $P=(x,t)$ a point in $\mathbb R^{n+1}$. For $P_0=(x_0, t_0)\in\mathbb R^{n+1}$ and $r>0$, denote
$$
\mathcal B_r(P_0):=\left\{(x,t)\in\mathbb R^{n+1}\Big||x-x_0|^2+(t-t_0)^2<r^2\right\}.
$$

Let $(a_{ij})_{n\times n}, a_0, b=(b_1, b_2,\cdots,b_n)$, and $c$ be given functions satisfying suitable properties to be specified later. Consider the operator
$$
\mathscr L\varphi=-a_{ij}\partial_{ij}\varphi+a_0\partial_t\varphi+b\cdot\nabla\varphi+c\varphi.
$$
Note that here $a_0$ is not required to have fixed sign and this linear operator can be regarded only as a linear degenerate elliptic operator in the space and time variables with degeneracy occurring in the time direction.

\begin{lemma}
  \label{lemma1}
Let $\mathcal O$ be a domain in $\mathbb R^{n+1}$. Assume that $a_{ij}, a_0, b,$ and $c$ are finitely valued functions in $\mathcal O$ with $c\geq0$, and the matrix $(a_{ij})_{n\times n}$ is nonnegative definite in $\mathcal O$. Then, for any
$\varphi\in C^{2,1}(\mathcal O)\cap C(\overline{\mathcal O})$, satisfying
\begin{equation*}
  \mathscr L\varphi>0 \mbox{ in }\mathcal O,\quad\mbox{and}\quad
  \varphi|_{\partial\mathcal O}\geq0,
\end{equation*}
it holds that $\varphi>0$ in $\mathcal O$. Here $C^{2,1}(\mathcal O)$ denotes the space of all function $f$ satisfying $f, \partial_tf, \nabla f, \nabla^2f\in C(\mathcal O)$.
\end{lemma}

\begin{proof}
First, we claim that $\varphi\geq0$ in $\mathcal O$. Otherwise, since $\varphi\geq0$ on $\partial\mathcal O$ and $\varphi\in
C(\overline{\mathcal O})$, there is $P_0\in\mathcal O$, such that $\varphi(P_0)=\min_{\overline{\mathcal
O}}\varphi<0.$ Since $\varphi\in C^{2,1}(\mathcal O)$, it is clear that $\partial_t\varphi(P_0)=\nabla \varphi(P_0)=0$
and $\nabla^2\varphi(P_0)$ is nonnegative definite. As a result
$$
(\mathscr L\varphi)(P_0)=-a_{ij}(P_0)\partial_{ij}\varphi(P_0)+c(P_0)\varphi(P_0)\leq0,
$$
which contradicts to the assumption. Therefore, the claim holds.
Next, we show that $\varphi>0$ in $\mathcal O$. Otherwise, there is $P_0^*\in\mathcal O$, such that $\varphi(P_0^*)=0$. Then, $\varphi(P_0^*)=\min_{\overline{\mathcal O}}\varphi=0,$ from which, similar as
before, one has $(\mathscr L\varphi)(P_0^*)\leq0$, contradicting to the assumption. Thus, $\varphi>0$ in $\mathcal O$, which proves the conclusion.
\end{proof}

\begin{lemma}[Hopf type lemma]
  \label{LEMHOPF}
Given $P_0=(x_0,t_0)$, $r>0$, $P_*=(x_*,t_*)\in\partial\mathcal B_r(P_0)$, $x_*\not=x_0$, and set $P_0^*=(x_0^*,t_0^*)$, with $x_0^*=\frac{x_0+x_*}{2}$ and $t_0^*=\frac{t_0+t_*}{2}$. Assume that there are positive constants $\lambda, \Lambda, \delta_*,$ and $C_*$, with $\delta_*<\frac{|x_0-x_*|}{4}$, such that
\begin{equation*}
  \left\{
  \begin{array}{l}
  \lambda|\xi|^2\leq a_{ij}(x,t)\xi_i\xi_j\leq\Lambda|\xi|^2,\quad\forall \xi\in\mathbb R^n,\\
  (t-t_0^*)a_0(x,t)+(x-x_0^*)\cdot b(x,t)\geq-C_*,\\
  0\leq c(x,t)\sqrt{|x-x_*|^2+(t-t_*)^2}\leq C_*,
  \end{array}
  \right.\quad\forall (x,t)\in\mathcal B_{\frac r2}(P_0^*)\cap\mathcal B_{\delta_*}(P_*).
\end{equation*}
Let $\varphi\in C^{2,1}(\mathcal B_r(P_0))\cap C(\overline{\mathcal B_r}(P_0))$ satisfy
$$
\mathscr L\varphi\geq0,\quad \varphi>\varphi(P_*),\quad \mbox{ in }\mathcal B_\frac{r}{2}(P_0^*)\cap \mathcal B_{\delta_*}(P_*),  \quad
\varphi(P_*)\leq0.
$$
Then, it holds that
$$
\varlimsup_{\ell\rightarrow 0^+}\frac{\varphi(P_*)-\varphi(P_*-\ell n_*)}{\ell}<0,
$$
where $n_*=\frac{P_*-P_0}r$ is the unit outward normal vector to $\partial\mathcal B_r(P_0)$ at $P_*$.
\end{lemma}

\begin{proof}
Set
$$
\mathscr D=\mathcal B_{\frac r2}(P_0^*)\cap\mathcal B_{\delta_*}(P_*).
$$
  It suffices to consider the case that $\varphi(P_*)=0$. Otherwise, one may consider $\Phi:=\varphi-\varphi(P_*)$, which reduces to the case considered, due to
  $$
  \mathscr L\Phi=\mathscr L\varphi-\mathscr L(\varphi(P_*))=\mathscr L\varphi-c\varphi(P_*)\geq\mathscr L\varphi\geq0 \quad\mbox{in }\mathscr D,
  $$
  as $\varphi(P_*)\leq0$ and $c\geq0$ in $\mathscr D$. It is clear that $\mathcal B_\frac r2(P_0^*)\subset\mathcal B_r(P_0)$. By assumption, it holds that
  \begin{equation}
    \label{ADD1224}
    \varphi(P)>\varphi(P_*)=0,\quad\forall P\in\overline{\mathscr D}\setminus\{P_*\}.
  \end{equation}
  Define
  $$
  \phi(x,t)=e^{-\zeta(|x-x_0^*|^2+(t-t_0^*)^2)}-e^{-\frac{r^2}{4}\zeta}=e^{-\zeta|P-P_0^*|^2}-e^{-\frac{r^2}{4}\zeta},
  $$
  where $\zeta>0$ is a constant to be determined.
  Then, it follows from direct calculations that
  \begin{align}
    \mathscr L\phi
    %=&-[4\zeta^2(x_i-x_{0i}^*)a_{ij}(x_j-x_{0j}^*)-2\zeta a_{ii}]e^{-\zeta|P-P_0^*|^2}\nonumber\\
%    &-2\zeta[(t-t_0^*)a_0+(x-x_0^*)\cdot b]e^{-\zeta|P-P_0^*|^2}+c\left(e^{-\zeta|P-P_0^*|^2}-e^{-\frac{r^2}{4}\zeta}\right)\nonumber\\
    =&-e^{-\zeta|P-P_0^*|^2}\Big[4(x-x_0^*)^TA(x-x_0^*)\zeta^2-2tr A\zeta\nonumber\\
    &+2((t-t_0^*)a_0+(x-x_0^*)\cdot b)\zeta+c\left(e^{\zeta(|P-P_0^*|^2-\frac{r^2}{4})}-1\right)\Big],\label{HF-1}
  \end{align}
  where $A=(a_{ij})_{n\times n}$ and $tr A=a_{ii}$. Note that the assumptions imply
  \begin{align}
    &4(x-x_0^*)^TA(x-x_0^*)\zeta^2-2tr A\zeta+2((t-t_0^*)a_0+(x-x_0^*)\cdot b)\zeta\nonumber\\
    \geq&4\lambda|x-x_0^*|^2\zeta^2-2n\Lambda\zeta-2C_*\zeta\geq\frac\lambda4|x_0-x_*|^2\zeta^2-(2n\Lambda+2C_*)\zeta, \label{HF-1-1}
  \end{align}
  for any $(x,t)\in\mathscr D$, due to $tr A\leq n\Lambda$ and
  $$
  |x-x_0^*|\geq|x_0^*-x_*|-|x_*-x|\geq\frac{|x_0-x_*|}{2}-\delta_*\geq\frac{|x_0-x_*|}{4},\quad\forall(x,t)\in\mathscr D.
  $$
  Note that $|P-P_0^*|<\frac r2$ for any $P\in\mathscr D$. It follows from the mean value theorem and the triangular inequality that
  \begin{align*}
    &\left|e^{\zeta(|P-P_0^*|^2-\frac{r^2}{4})}-1\right|=e^{\tau\zeta (|P-P_0^*|^2-\frac{r^2}{4})}\left||P-P_0^*|^2-\frac{r^2}{4}\right|\zeta\nonumber\\
    \leq&\left||P-P_0^*|-\frac{r}{2}\right|\left||P-P_0^*|+\frac{r}{2}\right|\zeta\leq r\zeta\big||P-P_0^*|-|P_0^*-P_*|\big|
    \leq r\zeta|P-P_*|,
  \end{align*}
  for any $P\in\mathscr D$, where $\tau\in(0,1)$. This, together with the assumptions, yields
  \begin{equation}
    \label{HF-1-2}
    \left|c\left(e^{\zeta(|P-P_0^*|^2-\frac{r^2}{4})}-1\right)\right|\leq cr|P-P_*|\zeta\leq C_*r\zeta,\quad\forall P\in\mathscr D.
  \end{equation}
  Combining (\ref{HF-1-1}) with (\ref{HF-1-2}) leads to
  \begin{eqnarray}
    &&4(x-x_0^*)^TA(x-x_0^*)\zeta^2-2tr A\zeta
\nonumber\\
     &&+2((t-t_0^*)a_0+(x-x_0^*)\cdot b)\zeta+c\left(e^{\zeta(|P-P_0^*|^2-\frac{r^2}{4})}-1\right)\nonumber\\
     &\geq&  \frac\lambda4|x_0-x_*|^2\zeta^2-(2n\Lambda+2C_*+rC_*)\zeta>0,\quad\forall P\in\mathscr D, \label{HF-2}
  \end{eqnarray}
  if $\zeta>\zeta_0:=\frac{8n\Lambda+8c_*+4rC_*}{\lambda|x_0-x_*|^2}$.
  Choose $\zeta=2\zeta_0$. Then, it follows from (\ref{HF-1}) and (\ref{HF-2}) that
  \begin{equation}
    \mathscr L\phi <0\quad\mbox{ in }\mathscr D.\label{HF-3}
  \end{equation}
  It follows from (\ref{ADD1224}) that
  \begin{eqnarray*}
  \varphi\geq0=\phi\mbox{ on }\partial\mathcal B_{\frac r2}(P_0^*)\cap\mathcal B_{\delta_*}(P_*),\quad   \inf_{\partial\mathcal B_{\delta_*}(P_*)\cap \mathcal B_{\frac r2}(P_0^*)}\varphi>0.
  \end{eqnarray*}
  Therefore, for $\varepsilon>0$ sufficiently small, it follows from the assumptions and (\ref{HF-3}) that
  \begin{equation*}
    \mathscr L\varphi\geq0>\mathscr L(\varepsilon\phi)\mbox{ in }\mathscr D, \quad \varphi\geq\varepsilon\phi\mbox{ on }\partial\mathscr D,
  \end{equation*}
  and thus
  \begin{equation*}
    \mathscr L(\varphi-\varepsilon\phi)>0\mbox{ in }\mathscr D, \quad \varphi-\varepsilon\phi\geq0\mbox{ on }\partial\mathscr D.
  \end{equation*}
  With the aid of this, noticing that $\varphi-\varepsilon\phi\in C^{2,1}(\mathscr D)\cap C(\overline{\mathscr D}),$ and applying Lemma \ref{lemma1}, one gets
  $$
  \varphi>\varepsilon\phi \quad\mbox{ in }\mathscr  D.
  $$
  Therefore, for $\ell>0$ sufficiently small, one has
  $$
  \varphi(P_*)-\varphi(P_*-\ell n_*)=-\varphi(P_*-\ell n_*)<-\varepsilon\phi(P_*-\ell n_*) =\varepsilon(\phi(P_*)-\phi(P_*-\ell n_*))
  $$
  and thus
  \begin{eqnarray*}
  \varlimsup_{\ell\rightarrow0^+}\frac{\varphi(P_*)-\varphi(P_*-\ell n_*)}{\ell}&\leq &\varepsilon\varlimsup_{\ell\rightarrow0^+} \frac{\phi(P_*)-\phi(P_*-\ell n_*)}{\ell} \\
  &=&\varepsilon\partial_{n_*}\phi(P_*)=-\varepsilon\zeta re^{-\frac{r^2}{4}\zeta}<0.
  \end{eqnarray*}
  This proves the conclusion.
\end{proof}

As a direct consequence of Lemma \ref{LEMHOPF}, the following corollary holds.

\begin{corollary}
  \label{CORHOPF}
Given $P_0=(x_0,t_0)$, $r>0$, $P_*=(x_*,t_*)\in\partial\mathcal B_r(P_0)$, $x_*\not=x_0$. Assume that $a_0, b, c\in L^\infty(\mathcal B_r(P_0))$, $c\geq0$ in $\mathcal B_r(P_0)$, and
\begin{equation*}
  \lambda|\xi|^2\leq a_{ij}(x,t)\xi_i\xi_j\leq\Lambda|\xi|^2,\quad\forall \xi\in\mathbb R^n, (x,t)\in\mathcal B_r(P_0),
\end{equation*}
for some positive constants $\lambda$ and  $\Lambda$.
Let $\varphi\in C^{2,1}(\mathcal B_r(P_0))\cap C(\overline{\mathcal B_r}(P_0))$ satisfy
$$
\mathscr L\varphi\geq0,\quad \varphi>\varphi(P_*),\quad \mbox{ in }\mathcal B_r(P_0),  \quad
\varphi(P_*)\leq0.
$$
Then, it holds that
$$
\varlimsup_{\ell\rightarrow 0^+}\frac{\varphi(P_*)-\varphi(P_*-\ell n_*)}{\ell}<0,
$$
where $n_*=\frac{P_*-P_0}r$ is the unit outward normal vector to $\partial\mathcal B_r(P_0)$ at $P_*$.
\end{corollary}

\subsection{Unboundedness of the entropy}

This subsection is devoted to proving Theorem \ref{UNBDDENTROPY}.
%In this case, we do no known if the temperature still have positive lower bounds after
%the initial time as in the previous section, but still we can see that the entropy becomes unbounded at positive time.
We start with the following embedding lemma, which is used
to verify the H\"older regularity of $J_y$ required in the proof of Theorem \ref{UNBDDENTROPY}.

\begin{lemma}
\label{Lem4.1}
Let $L>0$ be a positive number. Then, the following embedding inequality holds
$$
\|f\|_{C^{\frac12,\frac14}([-L,L]\times[0,T])}\leq C(\|f\|_{L^\infty(0,T; H^1((-L,L)))}+\|\partial_tf\|_{L^2(0,T; L^2((-L,L)))}),
$$
for any function $f\in L^\infty(0,T; H^1((-L,L)))$ such that $\partial_tf\in L^2(0,T; L^2((-L,L)))$, where $C$ is an
absolute positive constant.
\end{lemma}

\begin{proof}
For any $t,\tau\in[0,T]$, one deduces by Lemma \ref{Lem2.1}, the Minkovski, H\"older, and Cauchy inequalities that
\begin{eqnarray*}
  &&\|f(\cdot,t)-f(\cdot,\tau)\|_{L^\infty((-L,L))}\\
  &\leq& C\|f(\cdot,t)-f(\cdot,\tau)\|_{L^2((-L,L))}^\frac12\|f(\cdot,t)-f(\cdot,\tau)\|_{H^1((-L,L))}^\frac12\\
  &\leq& C\|f\|_{L^\infty(0,T; H^1((-L,L)))}^\frac12\left\|\int_\tau^t\partial_tf(\cdot,s)ds\right\|_{L^2((-L,L))}^\frac12\\
  &\leq& C\|f\|_{L^\infty(0,T; H^1((-L,L)))}^\frac12\left(\int_\tau^t\|\partial_tf\|_{L^2((-L,L))}ds\right)^\frac12\\
  %&\leq& C\|f\|_{L^\infty(0,T; H^1((-L,L)))}^\frac12\|\partial_tf\|_{L^2(0,T; L^2((-L,L)))}^\frac12|t-\tau|^\frac14\\
  &\leq&C(\|f\|_{L^\infty(0,T; H^1((-L,L)))}+\|\partial_tf\|_{L^2(0,T; L^2((-L,L)))}) |t-\tau|^\frac14,
\end{eqnarray*}
for an absolute positive constant $C$. For any $x,y\in[-L,L]$ and $t\in[0,T]$, it follows from the H\"older inequality that
\begin{eqnarray*}
  &&|f(x,t)-f(y,t)|\leq\left|\int_x^y\partial_xf(z,t)dz\right|\\
  &\leq&\left|\int_{-L}^L|\partial_xf|^2dx\right|^\frac12|y-x|^\frac12\leq\|f\|_{L^\infty(0,T; H^1((-L,L)))}|y-x|^\frac12.
\end{eqnarray*}
Therefore, for any $x,y\in[-L,L]$ and $t,\tau\in[0,T]$, it holds that
\begin{eqnarray*}
  &&|f(x,t)-f(y,\tau)|\leq|f(x,t)-f(y,t)|+|f(y,t)-f(y,\tau)|\\
  &\leq &C(\|f\|_{L^\infty(0,T; H^1((-L,L)))}+\|\partial_tf\|_{L^2(0,T; L^2((-L,L)))})(|t-\tau|^\frac14+|y-x|^\frac12).
\end{eqnarray*}
This leads to the conclusion.
\end{proof}

We are now ready to prove Theorem \ref{UNBDDENTROPY}.

\begin{proof}[Proof of Theorem \ref{UNBDDENTROPY}]
The proof is dived into five steps as follows.

\textbf{Step 1. Regularities and pointwise positivity of $\vartheta$. }For $L>0$, denote by $W^{2,1}_2((-L, L)\times(0,T))$ the space
of all functions
$f\in L^2(0,T; H^2((-L,L)))$ satisfying
$\partial_tf\in L^2(0,T; L^2((-L,L)))$. Recall the embedding that $W^{2,1}_2([-L,L]\times(0,T))\hookrightarrow C^{\frac12,\frac14}([-L,L]\times[0,T])$ (Theorem 4.1 of \cite{WUYINWANG}).
Note that
$$(v_y, \vartheta_y)\in W^{2,1}_2((-L, L)\times(0,T)),
\quad J, J_y\in L^\infty(0,T; H^1((-L,L))),$$
and $J_t, J_{yt}\in L^\infty(0,T; L^2((-L,L)))$. Hence, it follows from the Sobolev embedding
and Lemma \ref{Lem4.1} that
\begin{equation}
  (J, J_y, v_y, \vartheta_y)\in C^{\frac12,\frac14}([-L,L]\times[0,T]),\quad\forall L> 0.\label{CONT-1}
\end{equation}

Rewrite (\ref{Eqtheta}) as
\begin{equation}
  c_v\varrho_0\vartheta_t-\frac\kappa J\vartheta_{yy}+\kappa\frac{J_y}{J^2}\vartheta_y+R\varrho_0\frac{v_y}{J}\vartheta=\frac\mu J|v_y|^2.
  \label{CON-EQ}
\end{equation}
Since $J$ is uniformly positive on $\mathbb R\times(0,T)$, it can be checked
that all the coefficients in (\ref{CON-EQ}), i.e.,
$\varrho_0, \frac1J, \frac{J_y}{J^2}, \varrho_0\frac{v_y}{J},$ and $\frac{|v_y|^2} J$, belong to
$C^{\frac12,\frac14}([-L,L]\times[0,T])$.
Thanks to these and the fact that $\varrho_0(y)>0$ for all $y\in\mathbb R$, it follows from the classic
Schauder theory on interior regularities for uniform parabolic equations
that $\vartheta\in C^{2+\frac12, 1+\frac14}((-L,L)\times(0,T))$. On the other hand,
by the embedding theorem, it follows from the regularities of
$\vartheta$ that $\vartheta\in C([-L,L]\times[0,T])$. Therefore, it holds that
$$
\vartheta\in C^{2, 1}((-L,L)\times(0,T))\cap C([-L,L]\times[0,T]).
$$
Note that $\vartheta\not\equiv0$ on $\mathbb R\times(0,T)$. Otherwise, noticing that $\vartheta\in
C([0,T]; L^2(-L,L))$ for any $L>0$, one has $\vartheta_0\equiv0$; furthermore, it follows from (\ref{Eqtheta}) that
$v_y\equiv0$ on $\mathbb R\times(0,T)$, from which, since $v\in C([0,T]; L^2((-L,L)))$ for
any $L>0$, one has $v_0\equiv\text{Const.}$ This contradicts to the assumptions.
Therefore, one has $\vartheta\not\equiv0$ on $\mathbb R\times(0,T)$ and $\vartheta\geq0$.
Thanks to this and by the strong maximum principle, one gets
\begin{equation}
\label{REGTHETA}
0<\vartheta\in C^{2, 1}(\mathbb R\times(0,T))\cap C(\mathbb R\times[0,T]).
\end{equation}

\textbf{Step 2. Asymptotic behavior of $J_y$.} Note that
$$
J_{yt}=\frac1\mu(GJ_y+JG_y+R\varrho_0'\vartheta+R\varrho_0\vartheta_y)
$$
which implies
$$
J_y=\frac1\mu\int_0^te^{\frac1\mu\int_s^tGd\tau}(JG_y+R\varrho_0'\vartheta+R\varrho_0\vartheta_y)ds.
$$
Therefore
\begin{equation}
  \left|\frac{J_y(y,t)}{\varrho_0(y)}\right|\leq\frac1\mu e^{\frac1\mu\int_0^t\|G\|_\infty d\tau}\int_0^t\left(J\left|\frac{G_y}{\varrho_0}\right|+R\left|\frac{\varrho_0'}{\varrho_0}\right|\vartheta
  +R|\vartheta_y|\right)ds.\label{JY-1}
\end{equation}

For any $y\geq0$, it follows that
\begin{eqnarray*}
  \left|\frac{G_y(y,t)}{\varrho_0(y)}\right|&=&\left|\int_0^1\frac{G_y}{\varrho_0}dz+\int_0^1\int_z^y\left(\frac{G_y}{\varrho_0}
  \right)_y(z',t)dz'dz\right|\\
  &\leq&\int_0^1\left|\frac{G_y}{\varrho_0}\right|dz+\int_0^{y+1}\left|\left(\frac{G_y}{\varrho_0}\right)_y\right|dz\\
  &\leq&\left\|\frac{G_y}{\sqrt{\varrho_0}}\right\|_2(t)+\sqrt {y+1}\left\|\left(\frac{G_y}{\varrho_0}\right)_y\right\|_2(t)
\end{eqnarray*}
and
\begin{eqnarray}
  \vartheta(y,t)&=&\int_0^1\vartheta(z,t)dz+\int_0^1\int_z^y\vartheta_y(z',t)dz'dz\nonumber\\
  &\leq&\int_0^1\frac{\sqrt{\varrho_0}\vartheta}{\sqrt{\varrho_0}}dz+\int_0^{y+1}|\vartheta_y|dz
  \leq\frac{\|\sqrt{\varrho_0}
  \vartheta\|_2(t)}{\sqrt{\delta_0}}+\sqrt{y+1}\|\vartheta_y\|_2(t),\label{ASYTHETA}
\end{eqnarray}
where $\delta_0:=\inf_{[-1,1]}\varrho_0>0.$ Similar estimates hold also for $y<0$ and thus it holds for
any $y\in\mathbb R$ that
\begin{equation}
  \left|\frac{G_y(y,t)}{\varrho_0(y)}\right|\leq\left\|\frac{G_y}{\sqrt{\varrho_0}}\right\|_2(t)+\sqrt {|y|+1}\left\|\left(\frac{G_y}{\varrho_0}\right)_y\right\|_2(t)\label{JY-2}
\end{equation}
and
\begin{equation}
  \vartheta(y,t)
   \leq\frac{\|\sqrt{\varrho_0}
  \vartheta\|_2(t)}{\sqrt{\delta_0}}+\sqrt{|y|+1}\|\vartheta_y\|_2(t).\label{JY-3}
\end{equation}

Substituting (\ref{JY-2})--(\ref{JY-3}) into (\ref{JY-1}) and using (H1), one can get by the H\"older and Sobolev inequalities that
\begin{eqnarray*}
  \left|\frac{J_y(y,t)}{\varrho_0(y)}\right|&\leq&Ce^{C\int_0^t\|G\|_\infty ds}\int_0^t\Bigg(\left\|\frac{G_y}{\sqrt{\varrho_0}}\right\|_2+\sqrt{|y|+1}
  \left\|\left(\frac{G_y}{\varrho_0}\right)_y\right\|_2\Bigg)ds\\
  &&+Ce^{C\int_0^t\|G\|_\infty ds}\int_0^t\left(\|\sqrt{\varrho_0}\vartheta\|_2+\|\vartheta_y\|_2\sqrt{|y|+1}+\|\vartheta_y\|_{H^1}\right)ds\\
  &\leq&C \sqrt t\sqrt{|y|+1} \left[\int_0^t \left(\left\|\left(\frac{G_y}{\sqrt{\varrho_0}},\left(\frac{G_y}{\varrho_0}\right)_y\right)\right\|_2^2+\|\vartheta_y\|_{H^1}^2\right)ds\right]^\frac12\\
  &&+C t\sqrt{|y|+1} \|\sqrt{\varrho_0}\vartheta\|_{L^\infty(0,T; L^2)}  \\
  &\leq&C_1 \sqrt{|y|+1} ,
\end{eqnarray*}
that is,
\begin{equation}
  \left|\frac{J_y(y,t)}{\varrho_0(y)}\right|\leq C_1 \sqrt{|y|+1},\quad\forall y\in\mathbb R, t\in[0,T],\label{JYRHO}
\end{equation}
where the regularities of $(\vartheta, G)$ have been used.

\textbf{Step 3. A scaling transform. }Let $T>0$ be any arbitrary given constant. Assume by contradiction that $s\in L^\infty(\mathbb R\times(0,T))$. Since
$\vartheta=\frac ARe^{\frac{s}{c_v}}\left(\frac{\varrho_0}{J}\right)^{\gamma-1}$ and $J$ has uniform positive lower and upper bounds,
it follows that
\begin{equation}
  0\leq\vartheta(y,t)\leq C_T\varrho_0^{\gamma-1}(y),\quad\forall (y,t)\in\mathbb R\times[0,T]. \label{ZEROTHETA}
\end{equation}

Let $\beta>0$ to be determined later and introduce a scaling transform as
$$
f(y,t):=\vartheta(y^{-\beta},t),\quad y\in(0,\infty), t\geq0.
$$
Then, direct calculations yield
\begin{eqnarray*}
  &&\vartheta(y,t)=f(y^{-\frac1\beta},t),\\
  &&\vartheta_t(y,t)=f_t(y^{-\frac1\beta},t),\quad \vartheta_y(y,t)=-\frac1\beta y^{-(1+\frac1\beta)}f_y(y^{-\frac1\beta},t),\\
  && \vartheta_{yy}(y,t)=\frac1{\beta^2}y^{-(2+\frac2\beta)}f_{yy}(y^{-\frac1\beta},t)+\frac{\beta+1}{\beta^2}y^{-(2+\frac1\beta)}f_y(y^{-\frac1\beta},t),
\end{eqnarray*}
for any $(y,t)\in(0,\infty)\times(0,\infty)$. Besides, one deduces from (\ref{CON-EQ}) that
\begin{eqnarray}
  &&c_v\varrho_0(y^{-\beta})y^{-(2+2\beta)}J(y^{-\beta},t)f_t(y,t)-\frac{\kappa}{\beta^2}f_{yy}(y,t)\nonumber\\
  &&-\left(\frac{\kappa(\beta+1)}{\beta^2}y^{-1}
  +\frac\kappa\beta y^{-(1+\beta)}\frac{J_y(y^{-\beta},t)}{J(y^{-\beta},t)}\right)f_y(y,t)\nonumber\\
   &&+R\varrho_0(y^{-\beta})y^{-(2\beta+2)}v_y(y^{-\beta},t)f(y,t)\geq0,\qquad \label{EQ-f}
\end{eqnarray}
for all $(y,t)\in(0,\infty)\times(0,\infty)$.

\textbf{Step 4. Verifying conditions of Hopf type lemma. }Let $M_T$ be a positive constant to be determined later and define
$$
F(y,t)=e^{-M_Tt}f(y,t), \quad y\in(0,\infty), t\in[0,\infty).
$$
Due to (\ref{REGTHETA}), it is clear that
\begin{equation*}
  0<F\in C^{2,1}((0,\infty)\times(0,\infty))\cap C((0,\infty)\times[0,\infty)).
\end{equation*}
Moreover, it follows from (\ref{ZEROTHETA}) and (H2) that
\begin{eqnarray*}
  F(y,t)&=&e^{-M_Tt}\vartheta(y^{-\beta},t)\leq C_Te^{-M_Tt}\varrho_0^{\gamma-1}(y^{-\beta})\nonumber\\
  &\leq&C_TK_2^{\gamma-1}e^{-M_Tt}(1+y^{-\beta})^{-(\gamma-1)\ell_\rho}\leq C_TK_2^{\gamma-1}e^{-M_Tt}y^{(\gamma-1)\beta\ell_\rho},
\end{eqnarray*}
for an $y\in(0,\infty)$ and $t\in[0,\infty)$. Thus, one can define $F(0,t)=0$ for $t\in[0,\infty)$, such that $F$ is well defined
on $[0,\infty)\times[0,\infty)$, satisfying
\begin{equation}
\left\{
\begin{array}{l}
  F\in C^{2,1}((0,\infty)\times(0,\infty)) \cap C([0,\infty)\times[0,\infty)),\\
  F>0\mbox{ in }(0,\infty)\times(0,\infty), \quad F(0,t)=0,\quad\forall t\in[0,\infty).
  \end{array}
  \right.
  \label{REGF}
\end{equation}

It follows from (\ref{EQ-f}) that
\begin{equation}
  a_0(y,t)F_t(y,t)-aF_{yy}(y,t)+b(y,t)F_y(y,t)+c(y,t)F(y,t)\geq0,\label{EQ-F}
\end{equation}
in $(0,\infty)\times(0,\infty)$, where
\begin{eqnarray*}
a_0(y,t)&=&c_v\varrho_0(y^{-\beta})y^{-(2+2\beta)}J(y^{-\beta},t),\qquad a=\frac\kappa{\beta^2},\\
b(y,t)&=&-\left(\frac{\kappa(\beta+1)}{\beta^2}y^{-1}
  +\frac\kappa\beta y^{-(1+\beta)}\frac{J_y(y^{-\beta},t)}{J(y^{-\beta},t)}\right),\\
c(y,t)&=&\varrho_0(y^{-\beta})y^{-(2+2\beta)}\Big(c_vM_TJ(y^{-\beta},t)+Rv_y(y^{-\beta},t)\Big).
\end{eqnarray*}

Take arbitrary $t_0\in(0,T)$, $0<y_0<\min\{\frac12,t_0\}$, and set
\begin{eqnarray*}
P_0=(y_0,t_0), \quad r=y_0,\quad P_*=(0,t_0)=:(y_*,t_*),\quad\delta_*=\frac{y_0}{8},\\
 P_0^*=\left(\frac{y_0}{2}, t_0\right)=:(y_0^*,t_0^*),\quad
\mathscr D=\mathcal B_{\delta_*}(P_*)\cap \mathcal B_{\frac r2}(P_0^*).
\end{eqnarray*}
Then,
$$
P_*\in\partial\mathcal B_r(P_0),\quad \mathscr D=\mathcal B_{\frac{y_0}{8}}((0,t_0))\cap\mathcal B_{\frac{y_0}{2}}((\tfrac{y_0}{2},t_0)).
$$
For any $(y,t)\in\mathscr D$, due to $0<y<\frac{y_0}{8}<\frac1{16}$ and $\frac{t_0}{2}<t<\frac32 t_0$, one deduces
\begin{eqnarray}
  &&(t-t_0^*)a_0(y,t)+(y-y_0^*)b(y,t)\nonumber\\
  &=&c_v(t-t_0)\varrho_0(y^{-\beta})y^{-(2+2\beta)}J(y^{-\beta},t)\nonumber\\
  &&-\left(y-\frac{y_0}{2}\right)\left(\frac{\kappa(\beta+1)}{\beta^2}y^{-1}+\frac\kappa\beta
  y^{-(1+\beta)}\frac{J_y(y^{-\beta},t)}{J(y^{-\beta},t)}\right)\nonumber\\
  %&\geq&c_v(t-t_0)\varrho_0(y^{-\beta})y^{-(2+2\beta)}J(y^{-\beta},t)
%   -\frac\kappa\beta(y-\frac{y_0}{2})
%  y^{-(1+\beta)}\frac{J_y(y^{-\beta},t)}{J(y^{-\beta},t)} \nonumber\\
  &\geq&-c_vt_0\varrho_0(y^{-\beta})y^{-(2+2\beta)}J(y^{-\beta},t)-\frac{\kappa y_0}{2\beta}y^{-(1+\beta)}\frac{|J_y(y^{-\beta},t)|}{J(y^{-\beta},t)}\nonumber\\
  &\geq&- c_vt_0\overline j_T\varrho_0(y^{-\beta})y^{-(2+2\beta)}-\frac{\kappa}{\underline j_T\beta} y^{-(1+\beta)}|J_y(y^{-\beta},t)| ,\label{a0b}
\end{eqnarray}
where
\begin{equation}
\overline j_T:=\sup_{(y,t)\in\mathbb R\times[0,T]}J(y,t),
\quad \underline j_T:=\inf_{(y,t)\in\mathbb R\times[0,T]}J(y,t).\label{JT}
\end{equation}

Set
$M_T:=\frac{R\|v_y\|_{L^\infty(\mathbb R\times(0,T))}}{c_v\underline j_T}.$
Then,
\begin{equation*}
  c_vM_TJ(y^{-\beta},t)+Rv_y(y^{-\beta},t) \geq c_vM_T\underline j_T-R\|v_y\|_{L^\infty(\mathbb R\times(0,T))}=0
  \end{equation*}
  and
  \begin{eqnarray*}
  c_vM_TJ(y^{-\beta},t)+Rv_y(y^{-\beta},t)& \leq&  c_vM_T\overline j_T+R\|v_y\|_{L^\infty(\mathbb R\times(0,T))} \\
   &= &c_vM_T(\overline j_T+\underline j_T).
\end{eqnarray*}
Thus, for any $(y,t)\in\mathscr D$, since $\sqrt{|y-y_*|^2+|t-t_*|^2}\leq\delta_*\leq\frac{1}{16}$, it holds that
\begin{equation}
  0\leq c(y,t)\sqrt{|y-y_*|^2+|t-t_*|^2} \leq c_vM_T(\overline j_T+\underline j_T)\varrho_0(y^{-\beta})y^{-(2+2\beta)}.\label{estc-1}
\end{equation}
For any $(y,t)\in\mathscr D$, since $0<y<1$, it follows from (\ref{JYRHO}) and (H2) that
\begin{eqnarray}
  \varrho_0(y^{-\beta})y^{-(2+2\beta)} \leq K_2(1+y^{-\beta})^{-\ell_\rho}y^{-(2+2\beta)}\leq K_2y^{(\ell_\rho-2)\beta-2}\leq K_2
  \label{newadd-1}
\end{eqnarray}
and
\begin{eqnarray}
  y^{-(1+\beta)}|J_y(y^{-\beta},t)|&\leq& C_1y^{-(1+\beta)}\varrho_0(y^{-\beta})\sqrt{1+y^{-\beta}}\leq  C_1K_2y^{-(1+\beta)}(1+y^{-\beta})^{-\ell_\rho+\frac12}\nonumber\\
  &\leq& C_1K_2y^{(\ell_\rho-\frac32)\beta-1}\leq C_1K_2,\label{newadd-2}
\end{eqnarray}
as long as
$
\beta\geq\max\left\{\frac{2}{\ell_\rho-2},\frac{2}{2\ell_\rho-3}\right\}=\frac{2}{\ell_\rho-2}.
$

Due to (\ref{newadd-1}) and (\ref{newadd-2}), it follows from (\ref{a0b}) and (\ref{estc-1}) that
\begin{eqnarray}
  (t-t_0^*)a_0(y,t)+(y-y_0^*)b(y,t)\geq -(C_1+1)K_2\left(c_vt_0\overline j_T+\frac{\kappa}{\underline j_T\beta}\right),\label{PROPa0b}\\
  0\leq c(y,t)\sqrt{|y-y_*|^2+|t-t_*|^2} \leq c_vM_T(\overline j_T+\underline j_T)K_2,\label{PROPc}
\end{eqnarray}
for any $(y,t)\in\mathscr D$, as long as $\beta\geq\frac{2}{\ell_\rho-2}.$

\textbf{Step 5. Unboundedness of entropy.} Choose
$$
\beta=\beta_0:=\max\left\{\frac{2}{(\gamma-1)\ell_\rho} ,\frac{2}{\ell_\rho-2}\right\}.
$$
Due to (\ref{REGF}), (\ref{EQ-F}), (\ref{PROPa0b}), and (\ref{PROPc}), it follows from Lemma \ref{LEMHOPF} that
\begin{eqnarray*}
  \varlimsup_{\ell\rightarrow 0^+}\frac{F(P_*)-F(P_*-n_*\ell)}{\ell}=-\varliminf_{\ell\rightarrow0^+}
  \frac{F(P_*-n_*\ell)}{\ell}=-2\varepsilon_2,
\end{eqnarray*}
for some positive constant $\varepsilon_2$,
where we recall $P_*=(0,t_0)$, $n_*=\frac{P_*-P_0}{r}=(-1,0)$, and thus $P_*-n_*\ell=(\ell, t_0)$. Thus, there is a positive number $\ell_0$, such
that
\begin{equation*}
  F(y, t_0)=e^{-M_Tt_0}\vartheta(y^{-\beta_0},t_0)\geq\varepsilon_2y, \quad\forall y\in(0,\ell_0),
\end{equation*}
that is
\begin{equation*}
  \vartheta(y, t_0)\geq \varepsilon_2e^{M_Tt_0}y^{-\frac1{\beta_0}},\quad \forall y\in\left(\ell_0^{-\frac1{\beta_0}},\infty\right).
\end{equation*}

On the other hand, it follows from (\ref{ZEROTHETA}) and (H2) that
\begin{eqnarray*}
  \vartheta(y,t) \leq C_TK_2^{\gamma-1}(1+y)^{-\ell_\rho(\gamma-1)}
  \leq C_TK_2^{\gamma-1}y^{-\ell_\rho(\gamma-1)}, \quad\forall y>0.
\end{eqnarray*}
Combing the previous two inequalities leads to
$$
y^{(\gamma-1)\ell_\rho-\frac1{\beta_0}}\leq C_TK_2^{\gamma-1}\varepsilon_2^{-1}e^{-M_Tt_0},\quad  \forall y\in(\ell_0^{-\frac1{\beta_0}},\infty),
$$
which is impossible when $y\rightarrow\infty$, as $(\gamma-1)\ell_\rho-\frac1{\beta_0}\geq\frac{\gamma-1}{2}\ell_\rho>0$.
This contradiction leads to the desired conclusion that $s\not\in L^\infty(\mathbb R\times(0,T))$.
\end{proof}

\section{Uniform positivity of $\vartheta$ and asymptotic unboundedness of $s$}
\label{SECPOSTEM}
In this section, we prove the uniform positivity of the temperature and asymptotic unboundedness of the entropy, under the condition that the initial density decays at the far field not slower than $O(\frac1{x^4})$, which yields the proof of Theorem \ref{THMPOSTEM}.

\begin{proof}[Proof of Theorem \ref{THMPOSTEM}]
We need only to prove (i), while the conclusion (ii) follows from (i), (\ref{ENTROPY'}), and (\ref{JY-3}),
as $J$ has uniformly positive
lower and upper bounds at each time $t\in(0,\infty)$.

Let $h$ be the Kelvin transform of $\vartheta$ defined as
\begin{equation}\label{Defh}
h(y,t)=y\vartheta\left(\frac1y,t\right),\quad \forall y\not=0, t\in[0,T].
\end{equation}
Then (\ref{REGTHETA}) implies
\begin{equation}
 h\in C^{2,1}((\mathbb R_+\cup\mathbb R_-)\times(0,T))\cap C((\mathbb R_+\cup\mathbb R_-)\times[0,T]).
\label{REGh}
\end{equation}
Note that
\begin{align*}
  &\vartheta(y,t) = yh\left(\frac1y,t\right), \quad
  \vartheta_t(y,t) = yh_t\left(\frac1y,t\right), \\
  &\vartheta_y(y,t) = h\left(\frac1y,t\right)-\frac1yh_y\left(\frac1y,t\right),\quad \vartheta_{yy}(y,t) = \frac{1}{y^3}h_{yy}\left(\frac1y,t\right),
\end{align*}
for any $y\not=0$ and $t\in(0,T)$. It follows from these and (\ref{CON-EQ}) that
%\begin{align*}
%  c_v\varrho_0\left( y\right)yh_t\left(\frac1y,t\right)-\frac\kappa{J(y,t)}\frac1{y^3}h_{yy}\left(\frac1y,t\right)+\frac{\kappa J_y(y,t)}{J^2(y,t)}\left(h\left(\frac1y,t\right)-\frac1yh_y\left(\frac1y,t\right)\right)\\
%+ R\frac{\varrho_0(y)}{J(y,t)}yh\left(\frac1y,t\right)
%  v_y(y,t)=\frac\mu{J(y,t)}(v_y(y,t))^2,\quad\forall y\not=0,
%\end{align*}
%which gives that
\begin{eqnarray*}
  &&c_v\varrho_0\left(\frac1y\right)\frac1{y^4}h_t\left(y,t\right)-\frac{\kappa}{J\left(\frac1y,t\right)}  h_{yy}\left(y,t\right)-\kappa \frac{J_y\left(\frac1y,t\right)}{J^2\left(\frac1y,t\right)}\frac1{y^2}h_y\left(y,t\right)
\\
&&+ \left(R\frac{v_y\left(\frac1y,t\right)}{J\left(\frac1y,t\right)}\frac{\varrho_0\left(\frac1y\right)}{y^4}
  +\kappa\frac{ J_y\left(\frac1y,t\right)}{J^2\left(\frac1y,t\right)}\frac1{y^3}\right)h(y,t)
 =\mu\frac{\left|v_y\left(\frac1y,t\right)\right|^2}{y^3 J\left(\frac1y,t\right)},
\end{eqnarray*}
for $y\not=0$. Define $a_0, a, b,$ and $\tilde c$ as
\begin{eqnarray*}
&& a _0:=c_v\varrho_0\left(\frac1y\right)\frac1{y^4}, \quad a :=\frac{\kappa}{J\left(\frac1y,t\right)},\quad  b:=-\kappa \frac{J_y\left(\frac1y,t\right)}{J^2\left(\frac1y,t\right)}\frac1{y^2},
\\
&&\tilde c:= R\frac{v_y\left(\frac1y,t\right)}{J\left(\frac1y,t\right)}\frac{\varrho_0\left(\frac1y\right)}{y^4}
  +\kappa\frac{ J_y\left(\frac1y,t\right)}{J^2\left(\frac1y,t\right)}\frac1{y^3},\quad\forall y\not=0, t\in[0,T].
\end{eqnarray*}
Then, it holds that
\begin{equation}
\label{eqh}
\left\{
\begin{array}{l}
   a _0h_t- a h_{yy}+ bh_y+\tilde ch\geq0,\quad\mbox{ in }Q_T^+, \\
   a _0h_t- a h_{yy}+ bh_y+\tilde ch\leq0,\quad\mbox{ in }Q_T^-,
   \end{array}
   \right.
\end{equation}
where
$$
Q_T^+:=\mathbb R_+\times(0,T),\quad Q_T^-:=\mathbb R_-\times(0,T).
$$

Properties of $ a _0,  a ,  b,$ and $\tilde c$ are analyzed as follows. It follows from (\ref{CONT-1}) and the regularities of $\varrho_0$ and $J$ that
\begin{equation}
   a _0\in C(\mathbb R^+\cup\mathbb R_-), \quad  a ,  b, \tilde c\in C(Q_T^+\cup Q_T^-). \label{CONT}
\end{equation}
For $ a _0$, it follows from (H3) that
\begin{equation}
  0\leq  a _0(y) \leq\frac{c_vK_3}{(|y|+1)^4}\leq c_vK_3,\quad\forall y\not=0. \label{BDDA0}
\end{equation}
For $ a $, it holds that
\begin{equation}
  \lambda_T\leq  a (y,t)\leq\Lambda_T,\quad\forall y\not=0, t\in[0,T], \label{BDDA}
\end{equation}
where $\lambda_T=\frac{\kappa}{\overline j_T}$ and $\Lambda_T=\frac{\kappa}{\underline j_T},$ with $\underline j_T$ and $\overline j_T$ given by (\ref{JT}).

It follows from (\ref{JYRHO}) and (H3) that
\begin{equation}
  |J_y(y,t)|\leq C_1\varrho_0(y) \sqrt{|y|+1}\leq C_1K_3(|y|+1)^{-\frac72},\quad\forall y\in\mathbb R, t\in[0,T].
  \label{JY-4}
\end{equation}
This implies that
\begin{equation}
  | b(y,t)|\leq\frac{\kappa}{\underline j_T^2}\frac1{y^2}\left|J_y\left(\frac1y,t\right)\right|\leq\frac{\kappa}{\underline j_T^2}\frac1{y^2} C_1K_3\left(\frac{1}{|y|}+1\right)^{-\frac72}\leq \frac{C_1K_3\kappa}{\underline j_T^2},\label{BDDB}
\end{equation}
for any $y\not=0$ and $t\in[0,T]$. By (H3) and (\ref{JY-4}), one deduces
\begin{eqnarray}
  |\tilde c(y,t)|&\leq&\frac R{\underline j_T}\frac1{y^4}\frac{K_3}{\left(1+\frac{1}{|y|}\right)^4}\|v_y\|_\infty(t)
  +\frac\kappa{\underline j_T^2}\frac1{|y|^3}C_1K_3\left(1+\frac{1}{|y|}\right)^{-\frac72}\nonumber\\
  &\leq&\frac{RK_3}{\underline j_T}\|v_y\|_\infty(t)+\frac{\kappa}{\underline j_T^2}C_1K_3
  \leq C(\|v_y\|_{L^\infty(0,T; H^1)}+1),  \label{BDDC}
\end{eqnarray}
for any $y\not=0$ and $t\in[0,T]$.

Set
$$
N_T=\frac{2}{c_v}\left(\frac{R}{\underline j_T}\|v_y\|_{L^\infty(\mathbb R\times(0,T))}+\frac{\sqrt2\kappa C_1}{\underline j_T^2}\right)
$$
and define
\begin{equation}
\label{DefH}
H(y,t)=e^{-N_Tt}h(y,t), \quad\forall y\not=0, t\in[0,T].
\end{equation}
Due to (\ref{REGh}), it is clear that
\begin{equation}
H\in C^{2,1}((\mathbb R_+\cup\mathbb R_-)\times(0,T))\cap C((\mathbb R_+\cup\mathbb R_-)\times[0,T]).
\label{PH-1}
\end{equation}
Since $\vartheta>0$, it follows from (\ref{Defh}) and (\ref{DefH}) that
\begin{equation}
  \label{PH-2}
  H>0 \mbox{ in }Q_T^+\quad\mbox{ and }\quad H<0\mbox{ in }Q_T^-.
\end{equation}
By (\ref{JY-3}) and recalling the definitions of $h$ and $H$, one deduces
\begin{eqnarray*}
  |H(y,t)|&\leq&|y|\vartheta\left(\frac1y,t\right)\leq C(\|\sqrt{\varrho_0}\vartheta\|_{L^\infty(0,T; L^2)}+\|\vartheta_y\|_{L^\infty(0,T; L^2)})|y|\sqrt{1+\frac1{|y|}}\nonumber\\
  &\leq&C\sqrt{y^2+|y|},\quad \forall y\not=0, t\in[0,T].
\end{eqnarray*}
Thanks to this, it holds that
\begin{equation}
  \lim_{(y,\tau)\rightarrow(0,t)}H(y,t)=0,\quad\forall t\in[0,T].  \label{PH-3}
\end{equation}
It follows from direct calculations and (\ref{eqh}) that
\begin{equation}
\label{PH-4}
\left\{
\begin{array}{l}
  a_0H_t-aH_{yy}+bH_y+cH\geq0, \quad\mbox{in }Q_T^+,\\
  a_0H_t-aH_{yy}+bH_y+cH\leq0, \quad\mbox{in }Q_T^-,
  \end{array}
  \right.
\end{equation}
where
\begin{eqnarray*}
  c(y,t)&:=&\tilde c(y,t)+N_Ta_0(y,t)\\
  &=&\left(c_vN_T+R\frac{v_y\left(\frac1y,t\right)}{J\left(\frac1y,t\right)}
  +\frac\kappa{J^2\left(\frac1y,t\right)}\frac{J_y\left(\frac1y,t\right)}{\varrho_0\left(\frac1y\right)}y
  \right)\varrho_0\left(\frac1y\right)\frac1{y^4}.
\end{eqnarray*}
For any $y\in[-1,0)\cap(0,1]$ and $t\in[0,T]$, it follows from (\ref{JYRHO}) that
\begin{eqnarray*}
  c(y,t)&\geq&\left(c_vN_T-\frac{R}{\underline j_T}\|v_y\|_\infty(t)-\frac\kappa{\underline j_T^2}|y|C_1\sqrt{1+\frac{1}{|y|}}\right)\varrho_0\left(\frac1y\right)\frac1{y^4}\\
  &=&\left(c_vN_T-\frac{R}{\underline j_T}\|v_y\|_\infty(t)-\frac\kappa{\underline j_T^2}C_1\sqrt{|y|^2+|y|}\right)\varrho_0\left(\frac1y\right)\frac1{y^4}\\
  &\geq&\left(c_vN_T-\frac{R}{\underline j_T}\|v_y\|_\infty(t)-\frac{\sqrt2\kappa C_1}{\underline j_T^2}\right)\varrho_0\left(\frac1y\right)\frac1{y^4}\\
  &\geq&\left( \frac{R}{\underline j_T}\|v_y\|_\infty(t)+\frac{\sqrt2\kappa C_1}{\underline j_T^2}\right)\varrho_0\left(\frac1y\right)\frac1{y^4}
\end{eqnarray*}
and thus
\begin{equation}
  \label{PC}
  c(y,t)\geq0,\quad\forall y\in[-1,0)\cup(0,1], t\in[0,T].
\end{equation}

Define
\begin{equation}
\label{DefTH}
\widetilde H(y,t)=\left\{
                    \begin{array}{cl}
                      H(y,t),&\mbox{if }y>0,\\
                      0, &\mbox{if }y=0,   \\
                      -H(y,t),&\mbox{if }y<0,
                    \end{array}
                  \right.
\end{equation}
for all $t\in[0,T]$. Denote
\begin{eqnarray*}
 \Omega_-:=(-1,0)\times(0,T), \quad \Omega_+:=(0,1)\times(0,T), \\ \Omega:=\Omega_+\cup\Omega_-,\quad\Gamma:=\{0\}\times[0,T].
\end{eqnarray*}
Then, it follows from (\ref{PH-1})--(\ref{PH-4}) that
\begin{align}
  \widetilde H\in C^{2,1}(\Omega)\cap C(\overline\Omega),\quad\widetilde H>0\mbox{ in }\Omega, \quad\widetilde H|_\Gamma=0, \label{REGTH}\\
  \mathscr L\widetilde H=a_0\widetilde H_t-a\widetilde H_{yy}+b\widetilde H_y+c\widetilde H\geq0 \quad\mbox{in }\Omega,
  \label{EQTH}
\end{align}
with $a_0, a, b,$ and $c$ satisfying
\begin{equation}
\label{REGABC}
\left\{\begin{array}{l}
   a_0\in C((-1,1)\setminus\{0\})\cap L^\infty((-1,1)\setminus\{0\}), \\
  a, b, c\in C(\Omega)\cap L^\infty(\Omega), \quad\lambda_T\leq a\leq\Lambda_T, \quad c\geq0\mbox{ in }\Omega,
   \end{array}
  \right.
\end{equation}
which follows from (\ref{CONT})--(\ref{BDDA}), (\ref{BDDB})--(\ref{BDDC}), and (\ref{PC}).

For arbitrary $t_0\in(0,T)$, set
$$
P_*=(0,t_0), \quad P_0=(y_0, t_0), \quad r=y_0=\min\left\{\frac12, t_0, T-t_0\right\}.
$$
Then, it is clear that $n_*:=\frac{P_*-P_0}{r}=(-1,0)$. Let $\mathcal B_r$ be the space-time
ball of radius $r$ and centered at $P_0$. Thanks to (\ref{REGTH}), (\ref{EQTH}), and (\ref{REGABC}), it is clear that
$\widetilde H$ satisfies all the conditions in Corollary \ref{CORHOPF}, and thus Corollary \ref{CORHOPF} implies
\begin{eqnarray*}
  \varlimsup_{\ell\rightarrow0^+}\frac{\widetilde H(P_*)-\widetilde H(P_*-\ell n_*)}{\ell}=\varlimsup_{\ell\rightarrow0^+}\frac{ -\widetilde H(\ell,t_0)}{\ell}=-2\varepsilon_0,
\end{eqnarray*}
with a positive constant $\varepsilon_0$. Thus,
$\varliminf_{\ell\rightarrow0^+}\frac{\widetilde H(\ell,t_0)}{\ell}=2\varepsilon_0,$
which yields that
\begin{equation}
  \widetilde H(y,t_0)\geq \varepsilon_0y, \quad \forall y\in(0,\ell_0),
\end{equation}
for some positive constant $\ell_0$. Then, by the definition of $\widetilde H$, one derives
\begin{equation}
  \widetilde H(y,t_0)=e^{-N_Tt_0}h(y,t_0)=e^{-N_Tt_0}y\vartheta\left(\frac1y,t_0\right)\geq\varepsilon_0y,\quad\forall y\in(0,\ell_0)
\end{equation}
and thus,
\begin{equation}
  \vartheta(y,t_0)\geq\varepsilon_0e^{N_Tt_0}\geq\varepsilon_0,
  \quad\forall y\in\left(\frac1{\ell_0},\infty\right). \label{LBP}
\end{equation}

Similarly, there are positive constants $\varepsilon_1$ and $\ell_1$ such that
\begin{equation*}
  \vartheta(y,t_0)\geq\varepsilon_1,\quad\forall y\in\left(-\infty,-\frac1{\ell_1}\right).
\end{equation*}
Combining this with (\ref{LBP}) and recalling that $0<\vartheta\in C(\mathbb R\times[0,T])$, one has
\begin{equation*}
  \inf_{y\in\mathbb R}\vartheta(y,t_0)=\min\left\{\varepsilon_0, \vartheta_1, \inf_{y\in\left[-\frac1{\ell_1},\frac1{\ell_0}\right]}\vartheta(y,t_0)\right\}>0.
\end{equation*}
This yields the desired conclusion, and the proof of Theorem \ref{THMPOSTEM} is completed.
\end{proof}

\section*{Acknowledgments}{This work was supported by the Key Project of National Natural Science Foundation of China (Grant No. 12131010) and Guangdong Basic and Applied Basic Research Foundation (Grant No. 2020B1515310002).
The work of J. L. was also supported by the National
Natural Science Foundation of China (Grants No. 11971009 and 11871005) and by
the Guangdong Basic and Applied Basic Research Foundation (Grants No. 2019A1515011621 and
2020B1515310005).
The work of Z.X. was also supported by the Zheng Ge Ru Foundation and by the Hong Kong RGC Earmarked Research Grants (Grants No. CUHK-14301421,
CUHK-14300917, CUHK-14300819, and CUHK-14302819).}

\end{document}